\documentclass{amsart}
\usepackage{amsmath,amssymb,amsxtra,enumerate,graphics}
\usepackage[mathscr]{eucal}

\newtheorem{theorem}{Theorem}[section]
\newtheorem{lemma}[theorem]{Lemma}
\newtheorem{proposition}[theorem]{Proposition}

\newtheorem{definition}[theorem]{Definition}

\newcommand{\glt}[2]{G_{$1}\!\left($2\right)}

\newcommand{\Lie}{{\text{\rm Lie}}}

\def\gl{\mathfrak{gl}}
\def\sl{\mathfrak{sl}}
\def\newterm#1{{\em #1}}
\newcommand\B{{\mathsf{B}}}
\newcommand{\Extr}{{E}}
\newcommand{\PExtr}{{\mathcal E}}
\newcommand\FExtr{{\mathcal F}}

\newcommand\SExtr{{\mathcal S}}
\newcommand{\ProjSp}{{\mathbb P}}
\newcommand{\charac}{{\rm char}}

\newcommand{\PSL}{{\rm PSL}}
\newcommand{\ad}{{\rm ad}}

\newcommand{\gt}{ > }
\newcommand{\lt}{ < }
\newcommand{\isom}{\cong}

\def\comment#1{}

\begin{document}
\title[Inner ideals and buildings]{Inner ideals in Lie algebras and spherical buildings}

\author{Arjeh M.~Cohen}
\thanks{Dedicated to the memory of T.A.~Springer}

\date{\today}

\begin{abstract}
The correspondence found by Faulkner between inner ideals of the Lie algebra of a simple algebraic group and shadows on long root groups of the building associated with the algebraic group is shown to hold in greater generality (in particular, over perfect fields of characteristic distinct from two).

\end{abstract}

\maketitle

\section{Introduction}\label{S-intro}
Let $k$ be a field of characteristic $p$ (possibly $0$) distinct from $2$, and let $L$ be a Lie algebra over $k$. An element $x$ of $L$ is said to be \emph{extremal} if it is nonzero and  $[x,[x,L]]\subseteq k\cdot x$. Thus extremal elements are ad-nilpotent of index $2$ or $3$. If this index is $2$, that is, $\ad_x^2L = 0$, then $x$ is called a \emph{sandwich}, otherwise, it is called a \emph{pure extremal element} of $L$. More generally, a linear subspace $I$ of $L$  satisfying $[I,[I,L]]\subseteq I$ is called an \emph{inner ideal}. In particular, each $1$-dimensional inner ideal is spanned by an extremal element, and conversely. 

A point of the projective space $\ProjSp(L)$ on $L$ is called \emph{extremal} 
if it contains a pure extremal element of $L$.  We write $\Extr(L)$, or just
$\Extr$ if no confusion arises, for the set of all pure extremal elements
of $L$, and, similarly, $\PExtr$ or $\PExtr(L)$ for the set of all \emph{extremal
points}, which are $1$-dimensional subspaces of $L$ spanned by a member of $\Extr$. The notions introduced below are discussed in greater detail in \cite{CSUW,CohIvan,CohIvan2}.

Every pair $(x,y)$ of pure extremal elements of $L$ belongs to one of the five relations on $E$ described in the table below.
\[\begin{array}{|l|c|}
\hline
\text{\bf name}&\text{\bf condition}\\
\hline
\text{same point}&k\cdot x = k\cdot y\\
\text{strongly commuting}&k\cdot x +k\cdot  y\subseteq E\cup\{0\}\\
\text{polar}&[x , y]=0\ \text{and}\ k\cdot x +k\cdot  y\not\subseteq E\cup\{0\}\\
\text{special}&[x , y]\in E\\
\text{hyperbolic}&k\cdot x +k\cdot y+k\cdot [x,y]\cong \sl(k^2)\\
\hline
\end{array}
\]
Here, $\sl(V)$, for a vector space $V$, stands for the Lie algebra of linear maps $V\to V$ whose trace is zero. 
We will not only use the names of these five relations for pairs of  $\Extr$ but also for the corresponding pairs of points from $\PExtr$. If $(x,y)$ is strongly commuting, then the set of projective points of the span $\langle x,y\rangle = k\cdot x+k\cdot y$ lies entirely in $\PExtr$; such subsets of $\PExtr$ are called \emph{lines}; the collection of lines is denoted by $\FExtr$ or $\FExtr(L)$. A pair $(x,y)$ of pure extremal elements is said to be \emph{commuting} if and only if $x$ and $y$ represent the same point, are strongly commuting, or polar. If so, then $k\cdot x$ and $k\cdot y$ commute in $L$.

We find geometries from Lie algebras by taking $\PExtr$ as point set and introducing lines on the point set, which are members of $\FExtr$ or, if $\FExtr=\emptyset$, particular kinds of subsets of $\PExtr$ defined below. A pair of a set of points and a set of lines (which are viewed as sets of points) is called a \emph{point-line space}. A \emph{subspace} of such a space is a set $Q$ of points with the property that, whenever $x$ and $y$ are points of $Q$ lying on a common line, every point of the line also belongs to $Q$. Two points on a common line are called \emph{collinear}. Thus, $(\PExtr,\FExtr)$ is a point-line space in which pairs of points are collinear if and only if they are strongly commuting.

\begin{definition}\label{df-symp} \rm Suppose that $L$ has no sandwiches and is spanned by $E(L)$. If $(x,y)$ is a polar pair of pure extremal elements of $L$, then there is a unique subset $S(x,y)$ of $\PExtr$ defined as follows.
\begin{itemize}
\item If $\FExtr\ne\emptyset$, it is the \emph{convex closure of $k\cdot x$ and $k\cdot y$}, that is, the smallest subspace of the point-line space $(\PExtr,\FExtr)$ containing both $k\cdot x$ and $k\cdot y$ with the property that each point collinear with two non-collinear points of $S(x,y)$, also belongs to $S(x,y)$.
\item If $\FExtr=\emptyset$, it consists of all points commuting with all points that commute with both $k\cdot x$ and $k\cdot y$. 
\end{itemize}
The subset $S(x,y)$ is called the \emph{symplecton} generated by $x$ and $y$. It does not depend on the choice of a polar pair within $S(x,y)$.
Each pair of points of $S(x,y)$ is commuting. The collection of all symplecta is denoted by $\SExtr$ or $\SExtr(L)$.
\end{definition}

Buildings, in particular spherical buildings, were introduced by Tits \cite{Tits74} to study algebraic groups. See \cite{AbramenkoBrown} for an introduction. Here we will associate them with Lie algebras.
The following theorem shows how buildings arise from the point-line spaces we just defined. The (root) shadow space of a building (cf.~\cite{buek1}) is a kind of point-line space determined by the building. A non-degenerate polar space is a particular kind of shadow space of a building. These notions will be explained in Section \ref{S-geom}. The following theorem summarizes a sequence of results leading to the construction of a building associated with a simple Lie algebra generated by pure extremal elements.

\begin{theorem}[\cite{CSUW,CohIvan,CohIvan2,CuypersFleischmann,Cuypers2006,Cuyperspanhuis}]\label{RSSofL}
Suppose that $L$ is a finite-dimensional simple Lie algebra over $k$ which is generated by its pure extremal elements. Then $L$ has no sandwiches and is spanned by its pure extremal elements. Furthermore, there is a uniquely determined spherical building $\B(L)$ such that
\begin{itemize}
\item if $\FExtr\ne\emptyset$, then $\left(\PExtr,\FExtr\right)$ is the root shadow space of  $\B(L)$;
\item if $\FExtr=\emptyset$ and $\SExtr\ne\emptyset$, then $\left(\PExtr,\SExtr\right)$ is a non-degenerate polar space, and a shadow space of $\B(L)$. If, moreover, $\FExtr(L\otimes K)=\emptyset$ for each separable field extension $K$ of $k$ of degree $2$, then $\B(L)$ is associated with a symplectic group (i.e., an algebraic group of type $\mathsf{C}_n$).
\end{itemize}
\end{theorem}

The above theorem is applicable to algebraic groups over sufficiently large fields of characteristic distinct from $2$ in the following way. Let $G$ be a simple linear algebraic group defined over $k$ and let $K$ be a Galois field extension of $k$ with Galois group $\Gamma$ such that $G (k) = G(K)^\Gamma$, the group of fixed points of $G(K)$ under the action of $\Gamma$. We say that $K$ is a \emph{splitting field}  for $G$ over $k$ if $G(K)$ has a split maximal torus. By \cite[8.11, p.~117]{Borel}, there is a finite separable extension of $k$ that is a splitting field for $G$ over $k$. Let $\widetilde L = \Lie(G)$ be the Lie algebra of $G$ over $k$, and let $L_{G}$ be the quotient of $[\widetilde L,\widetilde L]$ by the unique maximal proper $G(k)$-invariant ideal of $\widetilde L$. Then $[\widetilde L,\widetilde L]\otimes K $ is spanned by pure extremal elements (the long root elements are extremal), and therefore, so is $L_G \otimes K$. Since $L_G\otimes K$  is simple, the above theorem applies to it. In general, the building $\B(L_G\otimes K)$ of the above theorem will be isomorphic to the building of $G(K)$. An exception occurs if $p=3$ for $G$ of type $\mathsf{A}_2$, since for $G(k) = \PSL(k^3)$, the building $\B\left(L_G\right)$ is of type $\mathsf{G}_2$, see the proof of Proposition \ref{lmIOrd3EltSpByE}. As a consequence, the subbuilding $\left(\B(L_G\otimes K)\right)^\Gamma$ of $\B(L_G\otimes K)$ will coincide with the building of $G(k)$ and does not depend on the choice of $K$ (see \cite{MWeiss} for details). Therefore, we will write $\B(L_G)$ for the building $\left(\B(L_G\otimes K)\right)^\Gamma$. This definition of $\B(L_G)$ agrees with the definition of Theorem \ref{RSSofL} if $L = L_G$ satisfies the hypotheses. In fact, since every Lie algebra $L$ satisfying the hypotheses of the above theorem is of the form $L_G$ for a suitable simple linear algebraic group $G$ over $k$ of automorphisms of $L$, all buildings $\B(L)$ of the theorem are obtained in this way. Conversely, for anisotropic algebraic groups $G$, the set $\Extr(L_G)$ is empty, so a splitting field for $G$ over $k$ will be needed to detect the point set $\PExtr$.

The above theorem implies that the building $\B(L)$ is uniquely determined up to isomorphism, but the question remains how to recover the objects of the building directly from $L$.  Earlier results of Benkart \cite{BenkartTAMS} and Faulkner \cite{Faulkner} in the case where $k$ is algebraically closed of characteristic $0$ indicate that the notion of inner ideal is well suited for this purpose. The explicit classification of inner ideals for simple Lie algebras related to algebraic groups under the same conditions on the field is given in \cite{Fernandez} and confirms these results. The sets $\PExtr$, $\FExtr$, and $\SExtr$ are shadows of objects of $\B(L)$ and we will see (in Section \ref{S-shadsAreInner}) that not only pure extremal points but also lines and symplecta span inner ideals of $L$. This leads us to the following extension of Faulkner's results. 

\begin{theorem}[Main Theorem]
\label{MainTheorem}
Let $\Lie(G)$ be the Lie algebra of a simple algebraic group $G$ defined over a field $k$ whose characteristic is distinct from $2$. Let $L_G$ be the unique nontrivial simple quotient of $[\Lie(G),\Lie(G)]$. If $\charac(k)=3$, assume that $G$ is not of absolute type $\mathsf{A}_2$.  Let $k_s$ be the separable closure of $k$ with Galois group $\Gamma$. Then the proper nontrivial inner ideals of $L_G$ are in bijective correspondence with the $\Gamma$-invariant shadows on $\PExtr(L_G\otimes k_s)$ of $\Gamma$-fixed flags of the building $\B(L_G\otimes k_s)$, where the inner ideal of $L_G$ is obtained from the shadow as the set of $\Gamma$-fixed points of the span in $L_G\otimes k_s$ of the shadow.
\end{theorem}
The result implies that, with the exception noted, each object of the building $\B(L_G)$ corresponds to a proper inner ideal of $L_G$ and that each proper nontrivial inner ideal of $L_G$ can be obtained as the intersection of inner ideals of $L_G$ corresponding to objects from a flag of $\B(L_G)$. Especially if $\PExtr(L_G)=\emptyset$, we need the identifications of $L_G$ with $(L_G\otimes k_s)^\Gamma$ and $\B(L_G)$ with the subbuilding of $\Gamma$-fixed objects in $\B(L_G\otimes k_s)$ to formulate and obtain the result. A version without the use of an extension field but with the assumption $\PExtr(L_G)\ne \emptyset$ is given in Theorem \ref{th-InnerIsShadow}.

Just like for linear algebraic groups, the splitting fields are needed to reveal the latent ad-nilpotent elements. At one extreme, if $L=L_G$ for an anisotropic simple algebraic group $G$, it has no proper nontrivial inner ideals at all. At the other extreme, for a split algebraic group $G$, the building of the group and of the corresponding Lie algebra $L_G$ has a type of maximal rank, and $\PExtr$ spans $L_G$. Since $k_s$ is a Galois extension of $k$ which is splitting for $G$ over $k$,  the types of the buildings  $\B(L_G)$ and  $\B(L_G\otimes k_s)$ need not be equal and are usually referred to as the \emph{relative type} and \emph{absolute type}, respectively, of $L_G$.

If $K$ is a splitting field of $G$ over $k$, then the Lie algebra $L_G\otimes K$  is spanned by pure extremal elements and the conditions of Theorem \ref{RSSofL} are satisfied for $L_G\otimes K$. This makes it possible to apply Theorem \ref{RSSofL} to $L_G\otimes K$ and use the point-line space $(\PExtr,\FExtr)$ if $\FExtr\ne\emptyset $ and $(\PExtr,\SExtr)$ if $\FExtr=\emptyset $ and $\SExtr\ne\emptyset $ for $L\otimes K$. If $\SExtr=\emptyset$ and $\FExtr=\emptyset$, then the absolute  type of $L_G$ is $\mathsf{A}_{1}$  and the geometry has little more to offer than the Moufang set $\PExtr$ associated with the projective special linear group $\PSL(k^2)$. But if, for instance,  $L_G$ has relative type $\mathsf{A}_{1}$ and absolute type $\mathsf{A}_{2d-1}$ for some $d\ge 3$, then $\PExtr(L_G\otimes K)\ne\emptyset$ and $\FExtr(L_G\otimes K)\ne\emptyset$, and although $L_G$ has no extremal elements, we can use the correspondence between inner ideals and shadows of flags for $L_G\otimes K$ to identify the objects of $\B(L_G)$ with a single class of inner ideals of $L_G$. 

The main theorem does not mention how incidence of flags of the building $\B(L_G)$ can be described in terms of inner ideals of $L_G$. In general, this can be accomplished by means of prescribed intersections of the inner ideals.

Fields of characteristic $2$ are excluded for several reasons. The most important one is that a suitable definition of inner ideal is lacking. Only if the dimension of the inner ideal is $1$ do we have a useful definition of extremal element; it involves the Premet identities (cf.~\cite{CSUW}). By the way, with this definition, Theorem \ref{RSSofL} is also valid if $p=2$. Another reason is that, besides the exception for $\mathsf{A}_2$ in characteristic $3$, we need to set aside the case of $\mathsf{D}_4$ for characterstic $2$, since the automorphism group of the Lie algebra has type $\mathsf{F}_4$, The current proof of the main theorem needs special care to incorporate the case $p=3$. A treatment of characteristic $2$ is expected to be much more elaborate.

In Section \ref{S-geom} we will recall the geometric interpretation and the known algebraic results that we need. In Section \ref{S-shadsAreInner} we prove that shadows span proper inner ideals and in Section \ref{S-InnerAreShads}  that  inner ideals are spanned by shadows. We also indicate how the proof in case the characteristic of $k$ is distinct from $2$ and $3$ can be shortened by use of Jordan algebra results.  We finish with the proof of the main theorem in Section \ref{S-conclusion}, using some known characterizations of root shadow spaces of exceptional Lie type.

In \cite{DillonJAlg,dillonProcAMS}, Dillon stated an extension of Faulkner's work to Kac-Moody Lie algebras and buildings of affine type. Since the proper shadow spaces of buildings of affine type are all spherical and  much of the Jordan algebra theory is valid for Lie algebra of infinite dimension (although the descending chain condition on inner ideals is often needed), the question rises whether the means of the current paper might work in establishing a correspondence to the spherical case for all characteristics of the field distinct from $2$.

I am grateful to Hans Cuypers,  Ralf K\"ohl, Antonio Fern\'andez L\'opez, Bernhard M\"uhlherr, and Richard Weiss for their interest. Special thanks go to Fern\'andez L\'opez for his comments leading to Lemmas \ref{lmFernandezEnd} and \ref{lmFernandezFront}.

\section{Lie algebras and geometry}\label{S-geom}
In 1974, Tits \cite{Tits74} published his geometric approach to algebraic groups, in which spherical buildings are the main structures for constructing and understanding linear algebraic groups.
More traditional views, such as projective spaces in case of the special linear groups, are obtained from buildings in the guise of  shadow spaces. We assume the reader to be familiar with buildings and refer those who are not to books like \cite{AbramenkoBrown,buek1,Shult,Weiss03}. Here we briefly discuss how the usual geometries of points and lines (such as the points and lines of a projective space in the above example) can be obtained from a building. Buildings are defined by means of Coxeter diagrams $M_n$. Here $n$ denotes the number of nodes of the diagram, usually referred to as the \emph{rank} of the building. In the case of simple algebraic groups, the corresponding buildings have types determined by their Dynkin diagrams; the information about root length is lost, so the difference between the types involving $\mathsf{B}$ and $\mathsf{C}$ disappears. If the type of a building is derived from a Dynkin diagram, the building is called \emph{spherical}. According to \cite{Tits74}, all spherical buildings whose types are connected Coxeter diagrams on at least three nodes, are known. Generally, they are obtained from simple algebraic groups, but there are some slight variations. A more recent result \cite{TitsWeiss} is that all spherical buildings  of rank $2$ with sufficient symmetry (the technical term being \emph{Moufang})  are also obtained from simple algebraic groups and some slight variations thereof.

There are different ways of defining a building of a given Coxeter type $M_n$. We will mostly work with the view of \cite{Tits74}, where a building is a multipartite graph, whose parts are indexed by the nodes of $M_n$. Thus, there are $n$ parts and adjacent vertices come from distinct parts. For a node $i$ of $M_n$, vertices from the part indexed by $i$ are called \emph{objects of type} $i$. Two objects are called \emph{incident} if they are equal or adjacent. \emph{Flags} are sets of objects each pair of which is incident. The \emph{type} of a flag is the set of nodes of $M_n$ occurring as types of an object in the flag. Two flags are called \emph{incident} if their union is also a flag of the building. For a set $J$ of nodes of $M_n$, the \emph{shadow} on $J$ of a flag $F$ of the building is the set of flags of type $J$ incident with $F$. If $j\in J$, then a  \emph{line of type} $j$ on $J$ is the shadow on $J$ of a flag whose type consists of all nodes of $M_n$ but $j$. Thus, if $J$ is a singleton, there is only one type of line. The shadow space on $J$ of a building of type $M_n$ is the pair $(P, {\mathcal L})$ consisting of
\begin{itemize}
\item the set $P$ of  \emph{points}: the flags of type $J$,
\item the set $\mathcal L$ of  \emph{lines}: the subsets of $P$ that are lines of type $j$ for some $j\in J$.
\end{itemize}
If a spherical building is associated with an algebraic group of Dynkin type $M_n$, where $M$ is one of $\mathsf{A}$, $\mathsf{B},\ldots, \mathsf{G}$,  then its shadow space of type $J$ is said to be of \emph{type} $M_{n,J}$, or, if $J = \{j\}$, of type $M_{n,j}$.  When labelling the Dynkin diagrams we follow the tables at the end of Bourbaki \cite{Bourb}.

In case of a projective space $\ProjSp(V)$, where $V$ is a vector space over a division ring $\mathbb{D}$ of dimension $n+1\ge2$, the associated simple algebraic group is $\PSL(V)$. It has type $\mathsf{A}_n$ and so does the building. For $j\in\{1,2,\ldots,n\}$, the objects of type $j$ of the building are the $j$-dimensional subspaces of $V$ and incidence is symmetrized inclusion. The shadow space of type $\mathsf{A}_{n,1}$ is the (points and lines of the) usual projective space $\ProjSp(V)$ and the shadow space of type $\mathsf{A}_{n,n}$ is the usual projective space of the dual vector space (defined over the opposite of $\mathbb{D}$). More generally, the points of the shadow space of type $\mathsf{A}_{n,j}$ for arbitrary $j$ are the linear subspaces of $V$ of dimension $j$; for this reason, it is often referred to as a \emph{Grassmannian} (of subspaces of $V$ of dimension $j$).
If $\mathbb{D} = k$ is a field, the shadow spaces of these buildings emerging from the Lie algebra of the algebraic group $\PSL(V)$ have type $\mathsf{A}_{n,\{1,n\}}$. Its point sets consist of the incident pairs of projective points and hyperplanes of the usual projective space. The lines of type $1$ are the sets of incident pairs $(p,H)$ where $H$ if a fixed hyperplane and $p$ varies over all points of a fixed line of $\ProjSp(V)$. Similarly,  the lines of type $n$ are the sets of incident pairs $(p,H)$ where $p$ if a fixed projective point and $H$ varies over all hyperplanes containing a fixed subspace of $V$ of codimension $2$.

More generally, we will be working with \emph{root shadow spaces}. Their types have the form $M_{n,J}$ where $J$ is the set of the nodes of the Dynkin diagram $M_n$ adjacent to the node that should be added to obtain the affine extension of the Dynkin diagram.  For all spherical buildings of connected type $ M_n$, with the exception of $ M= \mathsf{A}$, this set is a singleton. More concretely, the types of root shadow spaces of buildings with a connected type are $\mathsf{A}_{1,1}$,  $\mathsf{A}_{n,\{1,n\}}$ $(n\ge2)$, $\mathsf{B}_{n,2}$ $ (n\ge 2)$, $\mathsf{C}_{n,1}$ $(n\ge2)$, $ \mathsf{D}_{n,2}$ $(n\ge 4)$, $ \mathsf{E}_{6,2}$, $\mathsf{E}_{7,1}$, $\mathsf{E}_{8,8}$, $\mathsf{F}_{4,1}$, and $\mathsf{G}_{2,2}$.

In order to introduce properties of root shadow spaces, we recall from \cite{Shult, buek1} the following definitions regarding a point-line space. Such a space is called
\begin{itemize}
\item \emph{connected} if any two points $p$ and $q$ can be connected by a finite chain of points $p=p_0,p_1,\ldots, p_t=q$ such that $p_i$ and $p_{i+1}$ are collinear. Such a chain is called a \emph{path of length} $t$.
\item \emph{partial linear} if each pair of distinct points are on at most one line.
\item \emph{gamma} if, for each line $L$ and point $p$, the set of points of $L$ collinear with $p$ is either empty, a singleton, or all of $L$.
\item \emph{polar} if, for each line $L$ and point $p$, the set of points of $L$ collinear with $p$ is either a singleton or all of $L$.
\item \emph{thick} if each line has at least three points and each point lies on at least three lines.
\item \emph{non-degenerate} if no point is collinear with all points.
\item \emph{singular} if any two points are collinear.
\end{itemize}

The minimal length $t$ of a path between any two points $p$ and $q$ is called the \emph{distance} between $p$ and $q$. The \emph{diameter} of a connected point-line space is the maximal distance occurring between any two points $p$ and $q$.

As mentioned in the introduction, a \emph{subspace} $T$ of a point-line space is a subset of the space with the property that, whenever $p$ and $q$ are collinear members of $T$, each point of each line on $p$ and $q$ also belongs to $T$. A subspace is called \emph{convex} if, for any two points $p$ and $q$ of the subspace, all points of a shortest path from $p$ to $q$ also belong to the subspaces. The definition of a symplecton for $\FExtr\ne\emptyset$ given in the introduction agrees with this definition. The \emph{singular rank} of a point-line space is the maximal length of a chain of non-empty singular subspaces strictly containing a point. Thus, if the singular is $0$, there are no lines; if the singular rank is $1$, then lines are maximal singular subspaces. 

By results of Tits, Veldkamp and Buekenhout-Shult  (cf.~\cite{buek1}), thick non-degenerate polar spaces of singular rank at least $n-1\ge 2$ are known to be shadow spaces of type $\mathsf{B_{n,1}}$, $\mathsf{C_{n,1}}$, or $\mathsf{D_{n,1}}$. Such a polar space is said to be of \emph{rank} $n$. A thick non-degenerate polar space or rank $n$ is said to be of \emph{symplectic type} if it is the shadow space on $1$ of the building of a simple algebraic group of symplectic type (Dynkin type $\mathsf{C}_n$); in this case the point-line space is as described in Lemma \ref{ex:symplectic} below.

A \emph{parapolar space} is a thick connected partial linear gamma space having a collection of convex subspaces, called \emph{symplecta}, which are non-degenerate polar spaces of singular rank at least $2$ with the following two properties. 

\begin{itemize}
\item For each pair $(p,q)$ of points at mutual distance two, either there is a unique point, denoted $[p,q]$, collinear with both, or they are contained in a symplecton. 
\item Each line lies in a symplecton.
\end{itemize}
If two points at mutual distance two are always contained in a symplecton, then the parapolar space is called \emph{strong}.

If $L$ is finite-dimensional and simple (by which we mean to imply that it is not commutative) and contains a pure extremal element, then, with the exception of a single Lie algebra (of Witt type for $p=5$), it is spanned by pure extremal elements, and the conditions of Theorem \ref{RSSofL} are satisfied. If $\FExtr\ne\emptyset$, then according to this theorem the point-line space $(\PExtr,\FExtr)$ is the root shadow space of a spherical building; it is a parapolar space in which two points are collinear if and only if they are strongly commuting and in which every pair  $(k\cdot x, k\cdot y)$ of points at mutual distance two with more than one common collinear point lies in a symplecton, which coincides with the symplecton $S(x,y)$ of Definition \ref{df-symp}.  If  $\FExtr=\emptyset$ and  $\SExtr\ne\emptyset$, then $(\PExtr,\SExtr)$  is a non-degenerate polar space.

By \cite{CuypersRobertsShpectorov} for the simply laced case and \cite{Fleischmann} for all cases with $\FExtr\ne\emptyset$, the Lie algebra $L$ is uniquely determined by the corresponding building $\mathsf{B}(L)$ for all spherical buildings of rank at least 3.

The following observation is of importance for the proof of the main theorem; it shows that shadows of $\B(L)$ span commutative subspaces of $L$.

\begin{lemma}\label{lm-strong}For every spherical building of connected type, the shadow of every flag on the root shadow space is a convex subspace. Moreover, it is a strong parapolar space of diameter $2$.
\end{lemma}

\begin{proof}Each shadow of the root shadow space of a spherical building of type $M_{n}$, together with the subsets which are also shadows of flags, can be viewed as a smaller building. The type of the smaller building is the Coxeter diagram obtained from $M_n$ by removing the nodes that are types of objects from the flag. The fact that the shadow is a convex subspace is known to hold for all shadows; see \cite{BCN} for details. Inspection of the shadows in the particular case of a root shadow space shows that any two non-collinear points of the shadow belong to a symplecton which is fully contained in the shadow. Therefore, the shadow has diameter $2 $ and is a strong parapolar space.
\end{proof}

\label{S-Benkart}
A key result to the transition from algebra to geometry is the following result of Benkart's.

\begin{theorem}\label{InnerIdealComm}
Let $k$ be a field of characteristic distinct from $2$ and let $L$ be a finite-dimensional simple Lie algebra over $k$ generated by its pure extremal elements. 
For each proper inner ideal $I$ of $L$, we have $[I,I]=0$. Moreover, each nonzero element of $I$ is ad-nilpotent of index $3$.
\end{theorem}

\begin{proof}
The first part, $[I,I]=0$, is a special case of \cite[Lemma 1.13]{BenkartTAMS} because, by \cite[Corollary 4.5]{CSUW}, there are no sandwiches in $L$. It follows from $[I,[I,[I,L]]]\subseteq [I,I]=0$ that all members of $I$ are ad-nilpotent of index at most $3$. Since these Lie algebras have no sandwiches, the index of any non-trivial element of $I$ must be $3$.
\end{proof}

As we will see in the next section, this result implies that the set of extremal points of an inner ideal with lines induced from the projective space $\ProjSp(L)$ is a point-line space with properties as in Lemma \ref{lm-strong}.

\section{Spans of shadows are inner ideals}\label{S-shadsAreInner}
Throughout this section, $k$ will be a field of characteristic $p$ (possibly $0$) distinct from $2$ and $L$ will be a Lie algebra over $k$. 

The intersection of inner ideals of $L$ is an inner ideal and the full Lie algebra $L$ itself is an inner ideal. Therefore, we have the notion of an inner ideal \newterm{generated} by a subset $Y$ of $L$. We will denote this inner ideal by $I_Y$ and by $I_y$ if $Y = \{y\}$. A spanning set of $I_Y$ can be found constructively by iteratively adding elements $[a,[b,x]]$ to $Y$ with $a,b\in Y$ and $x\in L$.

In this section we will show that, if $L$ is a finite-dimensional simple Lie algebra over $k$ generated by its pure extremal elements, then the span of the shadow of each flag of $\B(L)$ is an inner ideal of $L$. We first consider subspaces of $L$ spanned by a strongly commuting pair from $\PExtr$. The result follows directly from \cite[Proposition 13.65]{Fernandez} (the restriction $p\ne3$ is not needed for its validity). Here we give a direct proof under more stringent conditions.

\begin{lemma}\label{InnerxyColl}Let $L$ be a  finite-dimensional simple Lie algebra over $k$ generated by its pure extremal elements and suppose that
$(x,y)$ is a strongly commuting pair of elements from $\Extr$. Write $I=k\cdot x+k\cdot y$.
For each $z\in L$, the product $[x,[y,z]]$ is either zero or belongs to $I\cap\Extr$.
In particular, $I$ is an inner ideal of $L$.
\end{lemma}

\begin{proof}
By \cite{CSUW} $L$ is spanned by $\Extr$ and there is a non-degenerate associative symmetric bilinear form $g$ on $L$ such that, for each $x\in \Extr$ and $y\in L$, 
\[ [x,[x,y]] = 2g(x,y) x.\]
Let $z\in L$.
The assumption that $x$ and $y$ strongly commute implies that each member of $I$ is either $0$ or lies in $E$. In particular, $x+y\in E$, so $[x+y,[x+y,z]]=2g(x+y,z)(x+y)\in I$. Furthermore, using $[x,y]=0$, we find
\[\begin{array}{rcl}[x+y,[x+y,z]]&=&[x,[y,z]]+[y,[x,z]]+2g(x,z)x+2g(y,z)y\\&=&2[x,[y,z]]+2g(x,z)x+2g(y,z)y .\end{array}\]
Thus $[x,[y,z]]=\frac{1}{2}[x+y,[x+y,z]]-g(x,z)x-g(y,z)y\in  I$. Since $[x,[x,L]]=k\cdot x$ and $[y,[y,L]]=k\cdot y$ are contained in $I$, this shows that $I$ is an inner ideal.  Moreover, the above observation gives $I\subseteq \{0\}\cup E$, so $[x,[y,z]]\in  \{0\}\cup (I\cap E)$.
\end{proof}

The next lemma concerns polar pairs in case $\FExtr\ne\emptyset$. 

\begin{lemma}\label{InnerxyPolar}
Let $L$ be a finite-dimensional simple Lie algebra over $k$ generated by $E$ and suppose $\FExtr\ne\emptyset$. If $(x,y)$ is a polar pair from $E$, then $I=\langle S(x,y)\rangle$ is an inner ideal. It is equal to $I_{\{x,y\}}$, the inner ideal generated by $x$ and $y$.
\end{lemma}

\begin{proof}
Let $z\in \Extr$ be such that $[x,[y,z]]$ is nonzero.
For the proof that $I$ is an inner ideal, it suffices to show that $[ x,[ y, z]]\in I$. 
Because $x$ and $y$ commute, we have $[x,[y,z]]=[y,[x,z]]$ by Jacobi's identity. Since this Lie product is nonzero, each of the pairs $(x,z)$ and $(y,z)$ is  not commutative and therefore either special or hyperbolic.

Suppose that $(y,z)$ is special. Then $[y,z]$ belongs to $\Extr$ and is collinear (that is, strongly commuting) with $y$, so, according to \cite[Lemma 2(iv)]{CohIvan2}, there is a pure extremal element $y'\in S(x,y)$ collinear with each of $x$, $y$, and $[y,z]$.
This implies that $x$ and $[y,z]$ are special (as their Lie bracket is nonzero, they are either hyperbolic or special, but they have mutual distance $2$ and so cannot be hyperbolic) and $[x,[y,z]]$ is a scalar multiple of $y'$. Since $y'$ belongs to $S(x,y)$, we conclude that $k\cdot [x,[y,z]]$ is a point of $S(x,y)$ and so lies in $I$.

In view of the symmetry, we are also done if the pair $x$, $z$ is special. The case remains where both pairs $(x,z)$ and $(y,z)$ are hyperbolic. 

Since $\FExtr\ne\emptyset$ and $g$ is non-degenerate, there is a line on each point (cf.~\cite[Lemma 4]{CohIvan2}).
Pick a line on $y$ in $S(x,y)$. It contains a point $v$ such that $(v,z)$ is special. If $v$ is not collinear with $x$, then there is a pure extremal element $w$ in $k\cdot v+k\cdot y$ that is collinear with $x$ and distinct from $v$ and $y$. Now $[x,[w,z]]\in k\cdot x+k\cdot w\subseteq I$ by Lemma \ref{InnerxyColl} and $[x,[v,z]]\in  S(x,v) =S(x,y)$ by the previous paragraph, so
$[x,[y,z]]$, being a linear combination of $[x,[w,z]]$ and $[x,[v,z]]$, also belongs to $I$.

We are left with the case where $(x,z)$ and $(x,y)$ are hyperbolic and each point $w$ of $\PExtr$ strongly commuting with both $x$ and $y$ lies inside $ z^{\perp_g}$, the hyperplane of $L$ of elements $p$ with $g(z,p) =0$, so $(w,z)$ is special.
Now consider again a point $b$ in $S(x,y)$ collinear with $y$ but not with $x$. We will write $x^\perp$ for the set of points of $\PExtr$ collinear with $x$. It is a subspace of $(\PExtr,\FExtr)$. Since the set $x^\perp\cap b^\perp$  of points collinear with both $x$ and $b$ is a geometric hyperplane of the subspace $x^\perp$ (cf.~\cite{buek1}) that differs from $x^\perp\cap y^\perp$, it has a point that is (not special and hence) hyperbolic with $z$. In particular, by the previous paragraph, $[x,[b,z]]$ belongs to $\langle S(x,b)\rangle=I$. Let $a$ be the element of $\Extr$ on the line $k\cdot b+k\cdot y$ of $S(x,y)$ that is collinear with $x$. Now $y$ is a linear combination of $a$ and $b$, while both $[x,[a,z]]$ and $[x,[b,z]]$ belong to $I$ (by Lemma \ref{InnerxyColl} for $a$), we conclude that $[x,[y,z]]$ belongs to $I$.

It remains to show that $ \langle S(x,y)\rangle=I_{\{x,y\}}$. Since the left hand side is an inner ideal containing $x$ and $y$, it contains the right hand side. In order to prove the other inclusion, we assume that $z$ is a pure extremal element collinear with both $x$ and $y$. Since the local diameter of $(\Extr,\FExtr)$ equals $3$ (that is, the diameter of the point-line space whose points are the lines on a point, say $q$, and whose lines are the sets of lines on $q$ contained in a singular space on $q$ of singular rank $2$), there is a pure extremal element $w$  in $y^\perp\cap z^\perp$ such that $(x,w)$ is special with $[x,w] = z$. For the same reason, there is a pure extremal element $u$ in $w^\perp$ such that $(y,u)$ is special and $[y,u]=w$. Now $z = [x,[y,u]]$ belongs to $I_{\{x,y\}}$. This shows that $x^\perp\cap y^\perp$ is contained in $I_{\{x,y\}}$. In particular, the convex closure of $\{x,y\}$  lies in $ I_{\{x,y\}}$. Since by definition $S(x,y)$ is this convex closure, we have $S(x,y)\subseteq I_{\{x,y\}}$, and so $\langle S(x,y)\rangle = I_{\{x,y\}}$.
\end{proof}

Next we consider commuting pairs in the symplectic case.
For a vector space $V$, we write $\gl(V)$ to denote the Lie algebra of all linear maps $V\to V$ and (as before) $\sl(V)$ for the Lie subalgebra of all maps with zero trace.
\begin{lemma}\label{ex:symplectic}
Let $L$ be a  finite-dimensional simple symplectic Lie algebra over $k$. To be more precise, let $V$ be a vector space over $k$ of dimension at least $2$ and $B$ a non-degenerate alternating form on $V$. Define
\[\gl(V)_B=\left\{X\in\gl(V)\mid B(u,Xv)=B(v,Xu)\phantom{xxx}\text{for all }u,v\in V\right\},\]
and, similarly, $\sl(V)_B = \sl(V)\cap \gl(V)_B$. Then $L$ is the quotient of $\sl(V)_B$ by the center.
For $a,b\in V$ we denote by $D_{a,b}$ the element of $\gl(V)_B$ determined by \[D_{a,b}(v) = B(v,b)a+B(v,a)b\phantom{xxx}\text{for all }v\in V\]
We then have the following identities,  where $X\in \gl(V)_B$ and $D_a = \frac12 D_{a,a}$. 
\[\begin{array}{rcl}
\left[D_a,\left[D_b,X\right]\right]&=&B(a,Xb)\cdot D_{a,b}\\
D_{a,b}&=&D_{a+b}-D_a-D_b\\
\left[D_a,D_b\right]&=&-B(a,b)\cdot D_{a,b}\\

\left[D_{a,b},\left[D_{a,b},X\right]\right]&=&2B(a,Xb)D_{a,b}-2B(Xa,a)D_b-2B(Xb,b)D_a\\
\left[D_{a},\left[D_{a,b},X\right]\right]&=&0\\
\end{array}\]
Each extremal element of $\sl(V)_B$ is of the form $D_a$ for some nonzero $a\in V$, where $D_a(v)=B(v,a)a$.
The polar space associated with $(V,B)$ is isomorphic to the point-line space $\left(\PExtr,\SExtr\right)$. 
For $a$, $b$ in $V$ with $B(a,b)=0$, put \[S(a,b) = \left\{k\cdot D_c\mid c\in k\cdot a + k\cdot b,\ c\ne0\right\}\] Then $S(a,b)=S(D_a,D_b)$ is a line of $(\PExtr(\sl(V)_B),\SExtr(\sl(V)_B))$ and the subspace
$\langle S(a,b)\rangle $ of $\sl(V)_B$ is the inner ideal spanned by $ k\cdot D_a$, $k\cdot D_b$, and $k\cdot D_{a+b}$. Consequently, the image of $S(D_a,D_b)$ in $L$ spans an inner ideal of $L$.
\end{lemma}

\begin{proof}
The identities are easily verified. The first identity with $a=b$  shows that $D_a$ is an extremal element of $\gl(V)_B$. It is known (see \cite[\S3]{CuypersFleischmann}) that each extremal element of $L$ is the image under the quotient map of an element of this form.

Let $P(V,B)$ be the polar space with point set $\ProjSp(V)$ whose lines are the sets of points of a $2$-dimensional subspace of $V$ on which $B$ vanishes (a \emph{singular subspace} of $V$ with respect to $B$).   Two points $k\cdot a$ and $k\cdot b$  of $\ProjSp(V)$ are collinear in this polar space if and only if $B(a,b) = 0$, and so
the third identity shows that this is equivalent to $[D_a,D_b]=0$, that is, the commuting of the pair $(D_a,D_b)$. Also, in the case where $B(a,b) = 0$, the sets $S(a,b)$ (defined in the statement) and $S(D_a,D_b)$ (cf.~Definition \ref{df-symp}) both describe the line in $\SExtr$ corresponding to the line on $k\cdot a $ and $k\cdot b$ in $P(V,B)$, so  they coincide and the map assigning $D_a$ to $a$ is an isomorphism  $P(V,B)\to \left(\PExtr,\SExtr\right)$ of point-line spaces. 

 The second identity shows that $\langle S(a,b)\rangle $ is spanned by $ k\cdot D_a$, $k\cdot D_b$, and $k\cdot D_{a+b}$. The first, fourth, and fifth identities together show that each element of the form $[D_c,[D_d,\gl(V)_B]]$ for $c,d,\in \{a,b,a+b\}$ is contained in $\langle S(a,b)\rangle $. This implies that $\langle S(a,b)\rangle $ is an inner ideal of $\gl(V)_B$ and of $\sl(V)_B$.
\end{proof}

\begin{proposition}\label{ShadowsInnerProposition}
Let $L$ be a finite-dimensional simple Lie algebra generated by $\Extr(L)$ and $T$ the shadow on $\PExtr(L)$ of a flag of $\B(L)$.
Then $\langle T\rangle$ is a commutative inner ideal of $L$. 
\end{proposition}

\begin{proof}
Write $I = \langle T\rangle$.  If $T$ is an extremal point, then, as has already been observed in the introduction, $I$ is an inner ideal. 

By Lemma \ref{lm-strong}, every pair of points from $T$ is commuting and, if $(x,y)$ is a polar pair in $T$, then the symplecton $S(x,y)$ is fully contained in $T$.
Therefore, $I$ is commutative and it suffices to show for pairs of extremal elements $x$ and $y$ that, 
\begin{itemize}\item if $(x,y)$ is collinear, then $[x,[y,L]]\subseteq\langle  x, y\rangle$ and
\item if $(x,y)$ is polar, then $[x,[y,L]]\subseteq \langle S(x,y)\rangle$.
\end{itemize}
Indeed, then we have $I_{x,y}\subseteq I$ for every pair $(x,y)$ of points of $T$.

The first statement is proved in  Lemma \ref{InnerxyColl}. In case $\FExtr\ne\emptyset$, the second statement is proved in Lemma \ref{InnerxyPolar}. 
If  $\FExtr(L\otimes K)\ne\emptyset$ for some separable extension $K$ of $k$, then, the same lemma gives that $I\otimes K$  is an inner ideal of $L\otimes K$ (for, then $L\otimes K$ is spanned by $E(L\otimes K)$  because $L$ is spanned by $E(L)$), from which it follows that $[I,[I,L]] \subseteq (I\otimes K)\cap L = I$, so $I$ is an inner ideal of $L$.
We are left with the case where $\FExtr(L\otimes K)=\emptyset$ for every separable extension $K$ of $k$. By Theorem  \ref{RSSofL} this implies that $L$ is of symplectic type, so $L$ is as in Lemma \ref{ex:symplectic}, from which it follows that $S(x,y)$  spans an inner ideal of $I$. 
We conclude that the proposition also holds in the symplectic case.
\end{proof}

In the proof of the converse statement, given in the next section, we will need more detailed information on the extremal elements of Lie algebras of orthogonal type. Therefore we end this section with the following lemma. Since the proof is straightforward, it is omitted.
\begin{lemma}\label{lmOrth} Let $B$ be a non-degenerate symmetric bilinear form on a vector space $V$ over a field $k$ of characteristic distinct from two. For $a,b,x\in V$, write \[ D_{a,b}(x) = B(x,a)b-B(x,b)a\]
This is an element of $\gl(V)_B$.
For any pair of vectors $a$, $b$ in $V$ and $X\in \gl(V)_B$, we have
\[\begin{array}{rcl}\left[D_{a,b},X\right]&=&D_{Xb,a}-D_{Xa,b}\\
\left[D_{a,b},\left[D_{a,b},X\right]\right]&=&B(a,b)\left(D_{a,Xb}+D_{b,Xa}\right) -2B(a,Xb)D_{a,b}
\\&&
+B(a,a)D_{Xb,b}+B(b,b)D_{Xa,a}\\
\end{array}\]
Thus, if $\langle a,b\rangle$ is a singular subspace of $V$, then $D_{a,b}$ is an extremal element of $\sl(V)_{B}$ with
$\left[D_{a,b},\left[D_{a,b},X\right]\right]=2B(Xa,b)D_{a,b}$. If $V$ has dimension at least $4$, then all extremal elements of the Lie algebra $\sl(V)_B$ have this form.
For  $a,b,c,d$ in $V$ and $X\in\gl(V)$, we have
\[\begin{array}{rcl}
\left[D_{a,b},D_{c,d}\right]&=&B(b,c)D_{d,a}+B(b,d)D_{a,c}+B(a,d)D_{c,b}+B(a,c)D_{b,d}\\
\left[D_{a,b},\left[D_{c,d},X\right]\right]&=&B(b,Xd)D_{c,a}-B(a,Xd)D_{c,b}-B(b,Xc)D_{d,a}+B(a,Xc)D_{d,b}\\
&&+B(b,d)D_{Xc,a}-B(a,d)D_{Xc,b}-B(b,c)D_{Xd,a}+B(a,c)D_{Xd,b}\\
\end{array}\]
\end{lemma}

\section{Inner ideals are spanned by shadows}\label{S-InnerAreShads}
Throughout this section, we let $k$ be a field of characteristic $p$ (possibly $0$) distinct from $2$. 
Our goal is the following converse of Proposition \ref{ShadowsInnerProposition}.

\begin{theorem}\label{th-InnerIsShadow} \label{thInnIdIsSh}
Suppose that $L$ is a finite-dimensional simple Lie algebra generated by $\Extr(L)$ and $I$ is an inner ideal of $L$. 
Assume that $k$ is separably closed. Then the set $\PExtr_I$ of all extremal points contained in $I$ is a shadow $T$ on $\Extr(L)$ of a flag of $\B(L)$ such that $I$ is the span of $T$.
\end{theorem}

The proof of this theorem consists of two main steps. The first is to show by algebraic methods that inner ideals are spanned by pure extremal elements. Once this is established, the set $\PExtr_I$ will be recognised as the shadow of a flag by means of geometric methods.
Since $L$ itself is known to be spanned by $\PExtr$ and one-dimensional inner ideals are spanned by elements of $\PExtr$, we may restrict ourselves to proper inner ideals of dimension at least $2$.

In order to establish that inner ideals are spanned by extremal elements, it suffices to show that inner ideals $I_x$ generated by an ad-ilpotent element $x$ of index $3$ (introduced at the beginning of Section \ref{S-shadsAreInner}). This is a consequence of the fact that each inner ideal is spanned by such elements (cf.~Theorem \ref{InnerIdealComm}).
For all characteristics except $2$ and $3$, this is proved in Lemma \ref{lmFernandezFront} below (following suggestions by Fern\'andez L\'opez).
\begin{lemma}\label{lmFernandezEnd}
Let $k$ be a field of characteristic distinct from $2$ and let $L$ be a finite-dimensional simple Lie algebra over $k$ generated by its pure extremal elements. 
Suppose that $a$ is an ad-nilpotent element of $L$ of index $3$ such that
\begin{enumerate}[(i)]
\item $a\in\ad_a^2(L)$, and
\item $\ad_{\ad_a^2(x)}^2 = \ad_a^2\ad_x^2\ad_a^2$ for each $x\in L$.
\end{enumerate}
Then the inner ideal $I_a$ of $L$ generated by $a$ is spanned by extremal elements of $L$.
\end{lemma}

\begin{proof}
By (i), there is an element $e\in L$ such that $a=\ad_{a^2}(e)$. Since $L$ is spanned by $\Extr(L)$, there is a finite subset $D$ of $\Extr$ such that  $a=\sum_{d\in D} d$. Consequently, $a=\sum_{d\in D}\ad_a^2( d)$. By (ii), for each $x\in L$, we have \[\ad_{\ad_a^2 (d)} (x)=  \ad_a^2\ad_d^2\ad_a^2 (x) =   2g(d,\ad_a^2 x)\ad_a^2 (d) \]
where $g$ is as in the proof of Lemma \ref{InnerxyColl}.
This shows that ${\ad_a^2 (d)} $ is extremal. We conclude that $a$ is a linear combination of the extremal elements $\ad_a^2( d)$, which belong to  $I_a$.
\end{proof}

For $p\ne2,3$, the conditions of the above lemma are satisfied:
\begin{lemma}\label{lmFernandezFront}
Let $k$ be a field of characteristic distinct from $2$ and $3$ and let $L$ be a finite-dimensional simple Lie algebra over $k$ generated by its pure extremal elements. 
Then, for each  ad-nilpotent element $a$ of $L$ of index $3$, the conditions of Lemma \ref{lmFernandezEnd} are satisfied, so $I_a$ is spanned by extremal elements of $L$. 
\end{lemma}

\begin{proof}
Condition (i) follows from a powerful result concerning the Jordan algebra $L_a$, obtained by defining the multiplication $\bullet$ on $L$ by setting $x\bullet y = [[x,a],y]$ and dividing out the ideal $\left\{\left.z\in L\,\right|\, \ad_a^2z = 0\right\}$. This Jordan algebra is nondegenerate (see ~\cite[Theorem 8.51 (1)]{Fernandez}) and, having finite dimension, satisfies the descending chain condition on inner ideals, so, by the Capacity existence theorem (cf.~\cite[Theorem 20.1.3]{McCrimmon} or \cite[Theorem 8.61 (ii)]{Fernandez}),
the Jordan algebra $L_a$ has a unit, that is, an element $e$, such that $x\bullet e = x \bullet e = x$ for each $x \in L_a$.  By \cite[Theorem 8.61 (ii)]{Fernandez}), this implies $a\in\ad_a^2(L)$.

Condition (ii) is a straightforward consequence of computations in the associative algebra of endomorphisms of the vector space $L$, which can be found in \cite[Lemma 4.4(3)]{Fernandez}.
\end{proof}

By the way, in the setting of the above lemma, even more is true:  according to \cite[Proposition 4.6]{Fernandez}, the inner ideal $I_a$ coincides with $\ad_a^2L$.

As a consequence of the lemma, for the proof that $I_x$ for a nilpotent element $x$ of $L$ is spanned by extremal elements, only the case of characteristic $3$ is left. Nevertheless, we prove the spanning result for all characteristics distinct from $2$ in Propositions \ref{lmIOrd3EltSpByE} and \ref{PropOrd3Exceptional}.

\begin{proposition}\label{lmIOrd3EltSpByE}
Let $L$ be a finite-dimensional simple Lie algebra over $k$ spanned by extremal elements and suppose that the type of $\B(L)$ is classical, that is, one of $\mathsf{A}_n,\ldots,\mathsf{D}_n$. If  $p=3$, also allow for the type to be $\mathsf{G}_2$. Assume that $k$ is separably closed.  If $x$ an ad-nilpotent element of $L$, then the inner ideal $I_x$ of $L$ generated by $x$ is spanned by pure extremal elements.
\end{proposition}

\begin{proof} 
In light of what has been said above, we may and will assume that $I_x$ is a proper inner ideal of $L$ and that $x$ is not extremal. By Theorem \ref{InnerIdealComm}, $x$ is ad-nilpotent of index $3$ (for, index $2$ would imply the existence of a sandwich) and $I_x$ is commutative. 

 $V$ be a finite-dimensional vector space over $k$. We will study the Lie algebras of classical type by means of the Lie algebra  \[\gl(V)_B=\left\{X\in\gl(V)\mid B(Xu,v) =-B(u,Xv)\quad \text{ for all } u,v\in V\right\}\] where $B$ is trivial in case the type of  $L$ is $\mathsf{A}_n$, a non-degenerate alternating form if the type is $\mathsf{C}_n$, and a non-degenerate symmetric form if the type is $\mathsf{B}_n$ or $\mathsf{D}_n$. The corresponding classical Lie algebra $L$ will then be the quotient of $\sl(V)_B$ by its center.

 Each ad-nilpotent element $x$ of $L$ has a nilpotent inverse image $X$  in $\sl(V)_{B}$ and the inverse image of the inner ideal $I_x$ of $L$ generated by $x$ is an inner ideal of $\sl(V)_{B}$ spanned by the center and the inner ideal $I_X$ of $\sl(V)_{B}$. Moreover, nilpotent inverse images of extremal elements of $L$ are extremal elements of  $\sl(V)_{B}$ and each nilpotent element will have a nilpotent inverse image, so it suffices to study nilpotent elements of $\sl(V)_B$. 
 
 Let $X$ be a nilpotent element of $\sl(V)_{B}$ of ad-nilpotency index $3$. Then $X$ has a Jordan decomposition $  n_1\cdot J_1+n_2\cdot J_2+n_3\cdot J_3$ on $V$. Here, $n_i$ is the number Jordan blocks $J_i$ of size $i$ for $i=1,2,3$. This is a consequence of the fact that the actions of $X$ on the exterior and the symmetric tensor of $V$ (one of which occurs in $L$ as an $X$-module) would have Jordan blocks of size greater than $3$ if $X$ would have Jordan blocks of size larger than $3$ on $V$ (see  \cite[Lemma 12.11]{UniNil}). We will analyse $I_X$ for each of the three types of bilinear form $B$ individually.
 
\medskip \noindent\emph{Type $\mathsf{A}_n$}. We take the dimension of $V$ to be $n+1$.  Assume that $n_3\gt0$. Then there is $v\in V$ with $X^2v\ne0$.  The restriction of $X$ to the linear span $J$ of $X^2v,Xv,v$ has matrix $E_{12}+E_{23}$ with respect to the given triple of vectors, where $E_{ij}$ stands for the $3\times 3$ matrix all of whose entries are zero, except for the $i,j$ entry, which is equal to $1$. Now, for the $3\times 3$ matrix $A=(a_{ij})_{1\le i,j\le 3}$ of an element of $\gl(V)$ leaving $J$ invariant on $X^2v,Xv,v$, we have
\[\begin{array}{rcl}\ad_X^2A&=&\small \begin{pmatrix}a_{31}&a_{32}-2\,a_{21}&a_{33}-2\,a_{22}+a_{11}\\ 0&-2\,
 a_{31}&a_{21}-2\,a_{32}\\ 0&0&a_{31}\\ \end{pmatrix},\\\ad_X^3 A&=& \small \begin{pmatrix}
0&-3\,a_{31}&3\,a_{21}-3\,a_{32}\\
 0&0&3\,a_{31}\\ 
 0&0&0
 \end{pmatrix}\end{array}\]
Choosing $a_{31}=1$, we see that the ad-nilpotency index of $X$  is $3$ only if  $p=3$. (This can also be derived from \cite{UniNil}.)
Suppose, therefore, $p=3$. If $n=2$, then $X$ is extremal, contradicting the assumptions. Thus, $n\gt 2$. 

Choose a vector $w$ in $V$ with $Xw=0$. Now $U=\langle J,w\rangle$ is $X$-invariant. Interpreting the matrix $E_{ij}$ for $i,j\in\{1,2,3,4\}$ with $i\ne j$ as the element of $\sl(V)$ given by the matrix with respect to $X^2v,Xv,v,w$, we find that $[\ad_X^2 E_{34},\ad_X^2 E_{41}] = E_{13}$, so the linear subspace $ \ad_X^2 E_{34}$ of $I_X$ is not commutative.
As a consequence, $I_X=\sl(V)$. We conclude that the case $n_3\gt0$ only occurs for $n=2$ and $p=3$ (as discussed in the Introduction),  in which case the automorphism group of $L$ is of type $\mathsf{G}_2$.

Suppose now $n_3=0$. Since $X$ is nonzero, we must have $n_2\gt0$. If $n_2=1$, then $X$ is extremal, so we may assume $n_2\ge2$. Taking vectors $v_1,w_1,\ldots,v_t,w_t\in V$, where $t=n_2$, such that $Xv_1$, $v_1,\ldots, Xv_t$, $v_t$ are linearly independent, we find the matrix of $X$ on the span $J$ of these $2t$ vectors to be $E_{1,2}+E_{3,4}+\cdots +E_{2t-1,2t} $ and straightforward computations with a square matrix $A=(a_{ij})$  of size $2t$ shows that
 $\ad_X^2A$ is a linear combination  of the extremal elements of $\sl(V)$ corresponding to $E_{ij}$ for $i\in \{1,3,\ldots,2t-1\}$ and $j\in \{2,4,\ldots,2t\}$.
The settles the case of type $\mathsf{A}_n$.

\medskip \noindent\emph{Type $\mathsf{C}_n$ $(n\ge2)$}.
Let $V$ be the natural module for a Lie algebra of type $\text{C}_n$, so $\dim(V) = {2n}$ and there is a non-degenerate alternating bilinear form $B$ on $V$ such that $L$ is the quotient by the center of the Lie algebra $\sl(V)_{B}$ of all linear maps $A:V\to V$ of zero trace with $B(Av,w) = -B(v,Aw)$ for all $v,w\in V$.  Now $X$ has a Jordan decomposition $n_1\cdot J_1+n_2\cdot J_2+n_3\cdot J_3$ on $V$, where $n_3=0$ if $p\ne3$. 
 This is a consequence of the fact that the action of $X$ on the symmetric square of $V$ (which is isomorphic to $L$ as an $X$-module) has Jordan blocks of size greater than $3$ if there would be Jordan blocks for $X$ on $V$ of larger sizes, see  \cite[Lemma 12.11]{UniNil}. Let $m$ be the nilpotency index of $X$ on $V$. We assume that $X$ is nonzero, so $m\in\{2,3\}$. 
 
 We claim that there is a vector $v$ in $V$ such that $v,Xv,\ldots,X^{m-1}v$ span a Jordan block of size $m$ with $B(Xv,v)\ne0$. To see this, observe that, by the choice of $m$ there is a vector $v$ such that $v,Xv,\ldots,X^{m-1}v$ span a Jordan block of size $m$. If $B(Xv,v)=0$, then, since $B$ is non-degenerate, there exists $w\in V$ with $B(Xv,w) = 1$. If $B(Xw,w)\ne0$ and $X^{m-1}w \ne0$, then replacement of $v$ by $w$ and $B(v,w) = 0$ yields the claim. Remains the case where  $B(Xw,w) = 0$ or $X^{m-1}w = 0$. Choose a nonzero $\lambda\in k$ such that $X^{m-1}w\ne -\lambda X^{m-1}v$ if $X^{m-1}w\ne0$ and such that $\lambda\ne -2B(Xv,w)/B(Xw,w)$ if  $B(Xw,w) \ne 0$ and $X^{m-1}w = 0$ (this is possible since $k$ has at least three elements). Then $u = v+\lambda w$ satisfies
 $B(Xu,u) = \lambda (2 +\lambda B(Xw,w))\neq0$ and $X^{m-1}u\ne0$, so the claim holds (with $u$ instead of $v$).

\noindent\emph{Case $m=2$}. Here, we have $n_3=0$. The Jordan decomposition of $X$ on $V$ can be chosen to be orthogonal. To see this, we first observe that, by the above claim, there is a vector $v$ in $V$ such that $X$ has a Jordan block of size $2$ on the span  $\langle Xv,v\rangle$ which is non-degenerate (that is, $B(Xv,v)\neq0$).  Considering the restriction of $X$ to the orthoplement of $\langle Xv,v\rangle$ with respect to $B$ and using the claim iteratively, we find that an $X$-invariant subspace on which $X$ has an orthogonal decomposition $n_2\cdot J_2$. The orthogonal complement of this sum in $V$ is annihilated by $X$.

For $ j=1,2,\ldots,n_2$ choose vectors $v_j$ in $V$ generating the distinct $X$-submodules of the orthogonal decomposition on which $X$ has non-degenerate Jordan blocks of size $2$. The subspaces $k\cdot v_j+k\cdot Xv_j$ are mutually orthogonal. Since $k$ is separably closed, we can scale the $v_j$ so that $ B(v_j,Xv_j)=1$. In terms of the elements $D_v$ and $D_{u,v}$  of $\sl(V)_{B}$ for $u,v\in V$ with $u,v\neq0$ and $B(u,v) = 0$ (defined in Lemma \ref{ex:symplectic} by $D_v(x) = B(x,v) v$ and $D_{u,v}(x) = B(x,u) v+ B(x,v) u$), we have $X = \sum_{j=1}^{n_2}D_{Av_j}$, so $X$ is a sum of commuting extremal elements. Moreover, since $[D_u,D_v] = 0$ if $B(u,v)=0$,  the identities of Lemma  \ref{ex:symplectic} give
 \[ [X,[X,A]] =\sum_{i,j=1}^{n_2} \left[D_{Xv_i},[D_{Xv_j}, A]\right] =\sum_{i,j=1}^{n_2} B(Xv_i,AXv_j)D_{Xv_i,Xv_j}  \]
Application of this identity with $A = D_{v_h,v_l}$ leads to $[X,[X,D_{v_h,v_l}]] = 2D_{Xv_h,Xv_l}$. This shows that the linear span $\langle D_{Xv_i,Xv_j}\mid i,j = 1,\ldots,n_2\rangle$ is contained in $I_X$. By the identities of Lemma \ref{ex:symplectic},  we conclude that this linear span is an inner ideal of $\sl(V)_{B}$. For $A = D_{v_l}$, we find
 \[ \begin{array}{rcl}[X,[X,D_{v_l}]] &=&\sum_{i,j} B(Xv_i,D_{v_l}Xv_j)D_{Xv_i,Xv_j} \\&=&\sum_{i,j} B(Xv_j,v_l)B(Xv_i,v_l)D_{Xv_i,Xv_j} \\&=& D_{Xv_l,Xv_l}\\&=&2D_{Xv_l}\end{array} \]
  so $X$, being a linear combination of these extremal elements, belongs to the linear span of the $D_{Xv_i,Xv_j}$. We conclude that $I_X$ coincides with $\langle D_{Xv_i,Xv_j}\mid i,j =  1,\ldots,n_2\rangle$, which is the linear span of all points contained in the singular subspace of the shadow space $(\PExtr,\SExtr)$ generated by $kXv_1,\ldots,kXv_{n_2}$.
  
\noindent\emph{Case $m=3$}. Here $\charac(k) = 3$. By the above claim, there is vector $v$ such that $X^2v,Xv,v$ is a Jordan block of $X$ with $B(v,Xv) = 1$. Since $B(X^2v,v) =-B(Xv,Xv) = 0$ and $B(X^2v,Xv) = -B(X^3v,v) = 0$, the vector $X^2v$ belongs to the radical of $\langle X^2v,Xv,v\rangle$.
 Since $B$ is non-degenerate and $B(v,Xv) = 1$, there is a vector $w$ in $V$ with $B(v,w) = B(Xv,w) = 0$ and $B(X^2v,w) = 1$. Because $B(v,X^2w) = 1$, the span of $X^2w,Xw,w$ is another Jordan block of size $3$ for $X$. It is disjoint from the first Jordan block since, if a nonzero linear combination of $X^2w$, $Xw$, $w$ would belong to $\langle X^2v,Xv,v\rangle$, also $X^2w$ would belong to it (simply take the appropriate power of the linear combination). But if $X^2w =\alpha v+\beta Xv+\gamma X^2v$ for certain $\alpha,\beta,\gamma\in k$, then we find the following contradiction.
 \[\begin{array}{rclclcl}
 1&=& B(v,X^2w)& =&\alpha B( v,v)+\beta B(v, Xv)+\gamma B(v,X^2v)&=&\beta \\
    0&=& B(Xw,X^2w)& =&\alpha B( Xw,v)+\beta B(Xw, Xv)+\gamma B(Xw,X^2v)&=& \beta \\
 \end{array}\]
 Thus, $v$ and $w$ span a $6$-dimensional $X$-submodule, say $U$, of $V$. The matrix of the restriction of $B$ to $U$ with respect to the basis
 $X^2v,Xv,v,X^2w,Xw,w$ is  
 $$M= \begin{pmatrix} 0&0&0&0&0&1\\  0&0&1&0&-1&0\\  0&-1&0&1&0&0\\  0&0&-1&0&0&
 0\\  0&1&0&0&0&s\\  -1&0&0&0&-s&0\end{pmatrix}$$
 where $s= B(Xw,w)$. A straightforward computation shows that the restriction of
 $[[X,[X,D_v]],[X,[X,D_w]]]$  to $U$ has matrix \[\begin{pmatrix}1&0&0&0&s&0 \\ 0&1&0&0&0&s\\ 0&0&1&0&0&0 \\ 0&1
 &0&-1&0&0 \\ 0&0&1&0&-1&0 \\ 0&0&0&0&0&-1\end{pmatrix}\] with respect to the above basis. Therefore, 
 the members $[X,[X,D_v] ]$ and $[X,[X,D_w] ]$  of $I_X$ have a non-central commutator, and so the image of the inner ideal $I_X$  in $L$ is not commutative. By Theorem \ref{InnerIdealComm} this implies that it coincides with $L$.
We conclude that, in all cases of type $\mathsf{C}_n$, the inner ideal of $L$ generated by an ad-nilpotent element of index $3$ is spanned by extremal elements.
 
 \medskip\noindent\emph{Type $\mathsf{B}_n$ $(n\ge2)$ and $\mathsf{D}_n$ $(n\ge4)$}
Let $n\ge2$ and let $V$ be the natural module for a Lie algebra of type $\mathsf{D}_n$ or $\mathsf{B}_n$, so $\dim(V)={2n}$ or ${2n+1}$, and there is a non-degenerate symmetric bilinear form $B$ on $V$ such that $L$ is the quotient by the center of the Lie algebra $\sl(V)_{B}$ of all linear maps $A:V\to V$ of zero trace with $B(Av,w) = -B(v,Aw)$ for all $v,w\in V$. Moreover, $B(Av,v) = -B(v,Av) = -B(Av,v)$, so $B(Av,v) = 0$ because $\charac(k)\ne2$.

\medskip\noindent\emph{Case $n_3=0$}. Suppose first that the nilpotency index of $X$ on $V$ equals $2$.  Then $B(Xv,Xv)=0$ for each vector $v$, because it is equal to $-B(v,X^2v) $ and $X^2=0$. We claim that there is a singular Jordan block of size $2$ for $X$. To see this, let $v$ be a vector such that $\langle v, Xv\rangle$ is a Jordan block of size $2$. Then, by the above observations, $Xv$ lies in the radical of this block. By non-degeneracy of $B$, there is a vector $w$ in $V$ such that $B(Xv,w) = 1$. Since $B(Xw,v) = -B(Xv,w) = -1$, it follows that $\langle w,Xw\rangle$ is a Jordan block as well. 
 If $B(v,v)=0$ or $B(w,w) = 0$, then   $\langle v, Xv\rangle$ or  $\langle w, Xw\rangle$ is singular and the claim holds, so suppose this is not the case. Take $\lambda\in k$, such that  $\lambda^2 B(v,v)+2\lambda B(v,w)+B(w,w)=0$ (it exists since $k$ is assume to be separably closed), and set $u=\lambda v + w$. Then $ B(u,u) = 0$.
Moreover, $Xu=0$ would imply $0 =B( Xu,w)= \lambda B(Xv,w)+B(Xw,w) = \lambda$, which contradicts $B(w,w)\ne0$. Therefore, $\langle u,Xu\rangle$ is a singular Jordan block. This settles the claim that $X$ has a singular Jordan block of size $2$.

We next claim that there is a second Jordan block of size $2$ which together with the first, spans a non-degenerate $4$-dimensional subspace of $V$. To establish this claim, we  start with a vector $v$ such that $\langle Xv,v\rangle$ is a singular Jordan block of size $2$ and take, similarly to the above, a vector $w\in V$ with $B(Xv,w) = 1$. As before, $B(Xw,w) = B(Xw,Xw) = 0$. After adding a suitable scalar multiple of $Xv$ to $v$, we can arrange for $B(v,w) = 0$.  If $B(w,w) \ne0$, we replace $w$ by $w'=\mu X v+w$, where 
$\mu = -\frac12 {B(w,w)}$, to get $B(w',w') = 2\mu +B(w,w)=0$. We retain $B(Xv,w') = \mu B(Xv,Xv)+B(Xv,w)=1$ and $B(v,w') =\mu B(v,Xv)+B(v,w) =0$, so we may assume $B(w,w) = 0$, which implies that $\langle Xw,w\rangle$ is a singular Jordan block for $X$. 

Straightforward computations show that the matrix $M$ of $B$ with respect to the basis $Xv,v,Xw,w$ is
\[ M = \begin{pmatrix} 0&0&0&1\\ 0&0&-1&0\\ 0&-1&0&0\\ 1&0&0&0\end{pmatrix}
\]
This proves the claim that $X$ has two singular Jordan blocks, which together span a $4$-dimensional non-degenerate subspace $U$ of $V$. 

The restrictions of $X$ and $D_{Xv,Xw}$. to $U$ coincide. We proceed by applying the second claim to $U^\perp$. Continuing this way, we find that, for some $t\le n/2$, there are $v_1,w_1,\ldots,v_t,w_t$ in $V$  such that, for each $i$,  the subspace $U_i$ with basis $Xv_i,v_i,Xw_i,w_i$  is non-degenerate, the matrix of the restriction of $B$ to $U_i$ is equal to $M$, and \[ X = \sum_{i=1}^{t} D_{Xv_i,Xw_i}\]
 We study $[X,[X,\sl(V)_{B}]]$, which must be part of $I_X$. Since each quadruple $Xv_i$, $Xw_i$, $Xv_j$, $Xw_j$ spans a singular subspace of $V$, the last identity of Lemma \ref{lmOrth} gives, for $A\in\gl(V)_B$,
\[\begin{array}{rcl} \left[X,\left[X,A\right]\right] &=&\sum_{i,j=1}^{t} \left[D_{Xv_i,Xw_i},\left[D_{Xv_j,Xw_j},A\right]\right]\\
&\in&\langle D_{Xv_i,Xv_j},D_{Xv_i,Xw_j},D_{Xw_i,Xw_j} \mid i,j=1,\ldots,t \rangle\\
\end{array}\]
Furthermore, the identities of Lemma \ref{lmOrth} with $X= D_{e,f}$ to vectors $a,b,c,d$ of $V$ such that $\langle a,b,c,d\rangle$ is singular
gives
\[\begin{array}{rcl}
\left[D_{a,b},\left[D_{c,d},D_{e,f}\right]\right]&=& \left(B(b,f)B(d,e)-B(b,e)B(d,f)\right) D_{c,a}\\
&&-\left(B(a,f)B(d,e)-B(a,e)B(d,f)\right) D_{c,b}\\&&-\left(B(b,f)B(c,e)-B(b,e)B(c,f))\right) D_{d,a}\\&&+\left(B(a,f)B(c,e)-B(a,e)B(c,f)\right) D_{d,b}
\end{array}\]

Repeated application of this formula gives that $\left[X,\left[X,D_{a,b}\right]\right]=2D_{Xa,Xb}$ whenever $a,b\in\left\{v_1,w_1,\ldots,v_t,w_t\right\}$.
For example, for $h\ne l$, we have
\[\begin{array}{rcl} \left[X,\left[X,D_{v_h,w_l}\right]\right] &=&\sum_{i,j=1}^{t} \left[D_{Xv_i,Xw_i},\left[D_{Xv_j,Xw_j},D_{v_h,w_l}\right]\right]\\
&=&\sum_{i,j\in\{h,l\}} \left[D_{Xv_i,Xw_i},\left[D_{Xv_j,Xw_j},D_{v_h,w_l}\right]\right]\\
&=& \left[D_{Xv_h,Xw_h},\left[D_{Xv_h,Xw_h},D_{v_h,w_l}\right]\right]+ \left[D_{Xv_l,Xw_l},\left[D_{Xv_l,Xw_l},D_{v_h,w_l}\right]\right]\\
&&+\left[D_{Xv_h,Xw_h},\left[D_{Xv_l,Xw_l},D_{v_h,w_l}\right]\right]+\left[D_{Xv_l,Xw_l},\left[D_{Xv_h,Xw_h},D_{v_h,w_l}\right]\right]\\
&=&0+0-D_{Xw_l,Xv_h}-D_{Xw_l,Xv_h}\\
&=&2D_{Xv_h,Xw_l}\\
\end{array}\]
This implies that the extremal elements  $D_{Xv_h,Xv_l}$, $D_{Xv_h,Xw_l}$, and $D_{Xw_h,Xw_l}$ all belong to $\left[X,\left[X,\gl(V)_B\right]\right]$.
Since each summand $D_{Xv_l,Xw_l}$ of $X$ lies in $\left[X,\left[X,\gl(V)_B\right]\right]$, so does $X$. This shows that the inner ideal $I_X$ coincides with the span of the extremal elements $D_{Xv_i,Xw_j} ,D_{Xv_i,Xv_j} , D_{Xw_i,Xw_j}$ for $ {i,j=1,\ldots,t}$.

\medskip\noindent
\emph{Case $n_3\gt0$}. Next assume that $X$ has a Jordan block $J = \langle X^2u,Xu,u\rangle$ of size $3$ for some $u$ in $V$. Then $B(Xu,u) = B(X^2u,Xu) =B(X^2u,X^2u) =0$ and $B(X^2u,u)=-B(Xu,Xu)$. Set $r = B(X^2u,u)$ and $s=B(u,u)$.
 If $r=0$, the vector $X^2u$ belongs to the radical of the restriction of $B$ to $J$. Since $B$ is non-degenerate, there is a vector $v\in V$ such that $B(v,X^2u) = 1$. Because $k$ has more than two elements, there is a nonzero scalar $r'$ in $k$ such that $X^2u+r' X^2v\ne0$. Now $u'=u+r' v$ satisfies $X^2u'\ne0$ and 
 $B(u',X^2u')=B(u+r'v,X^2u+r'X^4v) = B(u+r'v,X^2u)=r' B(v,X^2u)
 =r'\ne0$. This implies that, after replacing $u$ by $u'$, we may assume $r\ne0$. Since, $k$ is separably closed, we can rescale $u$, so that $r=1$. Finally, if $s\ne0$, we can replace $u$ by $u' = u-\frac{s}{2}Xu$ so that $B(X^2u',u') = B(A^2u,u-\frac{s}{2} Xu) = 1$ and $B(u',u') = B(u-\frac{s}{2}X^2u,u-\frac{s}{2}X^2u) = s -2\cdot\frac{s}{2}=0$. We conclude that we may assume $s=0$, so the matrix $M$ of the restriction of $B$ to $J$ with respect to the basis $X^2u,Xu,u$  is \[M = \begin{pmatrix}0&0&1\\ 0&-1&0\\ 1&0&0\end{pmatrix} \]
 In particular, the subspace $J$ is non-degenerate.

If $n_3\gt1$ or $n_3=1$ and $n_2\gt0$, then the characteristic of $k$ must be $3$. For otherwise, as indicated in \cite[Lemma 2.11]{UniNil}, the action of $\ad_X$ on $\sl(V)_B$ has a Jordan block of size greater than $3$. We will deal with the case of characteristic $3$ separately at the end of this proof. For now, we continue with the assumption that, even if $k$ has characteristic $3$, we have $n_3=1$ and $n_2 = 0$, so $n_1 =  2n-2$ if the type of $\B(L)$ is $\mathsf{B}_n$ and $n_1=2n-3$ if the type is $\mathsf{D}_n$. 
Since $X$ annihilates the orthogonal complement of $J$, it coincides with $ D_{X^2u,Xu}$. In general, this need not be an extremal element  of $\sl(V)_B$ since $B(Xu,Xu)=1$ (so $Xu$ is not a point of the geometry). In fact, by the second identity of Lemma \ref{lmOrth}, we have
\[\begin{array}{rcl}
[X,[X,A]] &=&\left[D_{X^2u,Xu},[D_{X^2u,Xu},A]\right] \\ &=&-2B(X^2u,AXu)D_{X^2u,Xu} + B(Xu,Xu) D_{AX^2u,X^2u} \\ &=&  2B(AX^2u,Xu)X +  D_{X^2u,AX^2u} 
\end{array} \]
If $\dim(V) = 3$, then $\sl(V)_B$ is the 3-dimensional Lie algebra consisting of all $A\in\gl(V)$ whose matrix with respect to $X^2u,Xu,u$ has the form
\[  \begin{pmatrix} -a_{33}&a_{12}&0\\ a_{32}& 0 & a_{12}\\ 0 & a_{32} & a_{33}\end{pmatrix}\]
In particular, $AX^2u=a_{32} Xu-a_{33} X^2u$, so $D_{X^2u,AX^2u}=a_{32}D_{X^2u,Xu}$ and $[X,[X,A]] $ is a scalar multiple of  $X=D_{X^2u,Xu}$, which shows that $X$ belongs to $\Extr$.

Suppose that $\dim(V)\gt 3$. Then there is a vector $v\in V$ with $B(v,v) = 1$ orthogonal to $\langle X^2u,Xu,u\rangle$, and 
$X = D_{X^2u,Xu}=\frac12\left( D_{X^2u,Xu+v }+ D_{X^2u,Xu-v}\right)$.
Consequently, $I_X $ coincides with the span of all $ D_{X^2u,w }$ for $w\in V$ with $B(X^2u,w) = 0$. Indeed, the above expression shows that $X$ belongs to the span, and, as we saw above, if $B(X^2u,w) = 0$, then  $ D_{u,w} X^2u=B(X^2u, u )w =w $, so
\[\begin{array}{rcl}  \left[X,[X,D_{u,w}]\right]&=& \left[D_{X^2u,Xu },[D_{X^2u,Xu }, D_{u,w}]\right]\\
&=&-2B(X^2u,D_{u,w}Xu)D_{X^2u,Xu}-B(Xu,Xu) D_{X^2u, D_{u,w}X^2u} \\
&=&2B(D_{u,w}X^2u,Xu)D_{X^2u,Xu}+D_{X^2u, w} \\
&=&2B(w,Xu)D_{X^2u,Xu}+D_{X^2u, w} 
\end{array} \]
For $w=Xu$, this gives $-D_{X^2u, Xu}$. We conclude that $X$ and each $D_{X^2u,w}$ with $B(X^2u,w) = 0$ lies in $[X,[X,\sl(V)_B]$. Hence $I_X = \langle D_{X^2u,w }\mid w\in B(X^2u,w)=0\rangle$.
In particular, $X^2u+v$ and $X^2u-v$ belong to $I_X$ and  $I_X$ is the linear span of $S(D_{X^2u,X^2u+v},D_{X^2u,X^2u-v})$, the symplecton of all lines on $k\cdot X^2u$.

\medskip\noindent
\emph{Proof of $n_3=0$ or $n_3=1$ and $n_2=0$ in case $p=3$.}\label{J3char3} 
We have a single non-degenerate Jordan block $J$ for $X$ on $V$ with basis $X^2u,Xu,u$ and want to prove that $X$ annihilates the orthogonal complement $J^{\perp_B}$. 
This is already known for characteristic distinct from $2$ and $3$, but the proof will be valid for all characteristics distinct from $2$. We assume that $X$ does not annihilate $J^{\perp_B}$ and derive a contradiction. Since the nilpotency index of $X$ on $V$ is at most $3$, we must have $n_3\gt 1$ or $n_3=1$ and $n_2\gt0$.

First assume $n_3\gt1$. As before, by the same argument as before, but now applied to $X$ on $J^{\perp_B}$, the latter subspace contains a second non-degenerate Jordan block, say $K$ with basis $X^2v, Xv,v$. Put $U = J+ K$. Now the matrix $M$ of $B$ with respect to the bases of $J$ and $K$ is
\[\begin{pmatrix}  M&0\\ 0&M\end{pmatrix}\]
where $M$ is the $3\times3$ matrix displayed in Case $n_3>0$ of the proof above.
A straightforward computation shows that
\[\left[ \left[X,[X,D_{u,Xv}]\right],\left[X,[X,D_{v,Xv}]\right]\right] = D_{X^2u,X^2v}\]
so $I_X$, which contains both  $ \left[X,[X,D_{u,Xv}]\right]$ and $\left[X,[X,D_{v,Xv}]\right]$, is not commutative, contrary to our assumptions.

The case where $n_3=1$ and $n_2\gt0$ remains. By the analysis of the case where $n_3=0$ applied to $J^{\perp_B}$, there are vectors $v$ and $w$ such that
$Xv,v,Xw,w$ is a non-degenerate $4$-dimensional $X$-invariant subspace of $J^{\perp_B}$ on which $X$ has two Jordan blocks of size $2$ and such that
the matrix $M$ of the restriction of $B$ to $7$-dimensional subspace $U$ of $V$ with basis $X^2u,Xu,u,Xv,v,Xw,w$ is 
\[M = \begin{pmatrix} 0&0&1&0&0&0&0\\ 0&-1&0&0&0&0&0\\ 1&0&0&0&0&0&0\\
0&0&0&0&0&0&1\\ 0&0&0&0&0&-1&0\\ 0&0&0&0&-1&0&0\\ 0&0&0&1&0&0&0\end{pmatrix}\]
A straightforward computation shows that
\[\left[ \left[X,[X,D_{u,Xu}]\right],\left[X,[X,D_{u,w}]\right]\right] =-2 D_{X^2u,Xw}\]
so $I_X$, which contains both  $ \left[X,[X,D_{u,Xu}]\right]$ and $\left[X,[X,D_{u,w}]\right]$, is not commutative, contrary to our assumptions. This ends the poof that, even if $k$ has characteristic $3$, the inner ideal $I_X$  is only proper if either $n_3=0$ or $n_3=1$ and $n_2=0$.

This ends the proof of Proposition \ref{lmIOrd3EltSpByE}.
\end{proof}

In order to derive the same result for Lie algebras of exceptional type $M_n$, that is, one of $\mathsf{E}_n$ ($n=6,7,8)$, $\mathsf{F}_4$, $\mathsf{G}_2$, we use a modest part of the classification of the orbits of ad-nilpotent elements of index $3$ under the algebraic group $G$. Several authors contributed to it. The complete classification can be found in Liebeck-Seitz \cite{UniNil}. 

\begin{proposition} \label{PropOrd3Exceptional}
Let $L$ be a finite-dimensional simple Lie algebra over $k$ spanned by pure extremal elements  and  suppose that the type of $\B(L)$ is exceptional. Assume that $k$ is separably closed. If $x$ is a nilpotent element of $L$, then the inner ideal $I_x$ of $L$ generated by $x$ is spanned by pure extremal elements.
\end{proposition}

\begin{proof}
Let $x$ be a nilpotent element of $L$. By Theorem \ref{InnerIdealComm}, the inner ideal $I_x$ generated by $x$ coincides with $L$ if $\ad_x$ is nilpotent of index greater than $3$ or $\ad_x^2L$ is not commutative. We will use the classification of nilpotent elements, specifically tables 22.1.1--22.1.5 of \cite{UniNil}. Nilpotent elements of a simple Lie algebra of exceptional type are classified by use of Levi subalgebras in which they are distinguished. If the type of the Levi subalgebra has a component distinct from $\mathsf{A}_1$ and, if  $p=3$, also be distinct from $\mathsf{A}_2$, then the index is greater than $3$, so the representatives of conjugacy classes of ad-nilpotent elements of index at most $3$ are among those from the mentioned tables for which the corresponding Levi subalgebra has a type all of whose components are $\mathsf{A}_1$ (or, if the characteristic of $k$ is three, $\mathsf{A}_2$).

In the table below, representatives of the classes of the nilpotent elements of type $\mathsf{A}_1^r$ which are not long root elements (excluded because they are pure extremal) are listed  along with a representative in $L$ and the dimensions of $[x,L]$ and $\ad_x^2L = [x,[x,L]]$. Here, $\alpha_1$ is a fundamental long root in case $\mathsf{F}_4$ and a fundamental short root in case $\mathsf{G}_2$ (consistent with the notation of \cite{Bourb}). The type $\mathsf{A}_1$ refers to a long root element, and the adorned symbol  $\widetilde{\mathsf{A}}_1$ to a short root element.

If  $\ad_x^2L $ is not commutative, this is indicated in the last column by "nc". Otherwise, $\ad_x^2L$ is found to coincide with $I_x$ and spanned by extremal elements; in this case the last column indicates the type of shadow space generated by the extremal points contained in it. This suffices for the proof of the proposition in case the characteristic is distinct from $2$ and $3$. Below, we first give some more comments on an individual case and next deal with the case where $k$ has characteristic $3$.

 \[
\begin{array}{lcrrlc}
 \text{type}&\text{representative }x&\dim(\ad_xL)&\dim(\ad_x^2L)&I_x&\\
\hline\hline
\multicolumn6{c}{\text{Lie algebra of type }\mathsf{E}_6}\\
\hline
\mathsf{A}_1^2&X_{\alpha_2}+X_{\alpha_3}&32&8&\text{symp}&\\
\mathsf{A}_1^3&X_{\alpha_2}+X_{\alpha_3}+X_{\alpha_5}&40&13&\text{nc}&\\
\hline\hline
\multicolumn6{c}{\text{Lie algebra of type }\mathsf{E}_7}\\
\hline
\mathsf{A}_1^2&X_{\alpha_1}+X_{\alpha_2}&52&10&\text{symp}&\\
\left(\mathsf{A}_1^3\right)^{(1)}&X_{\alpha_2}+X_{\alpha_5}+X_{\alpha_7}&54&27&\mathsf{E}_{6,1}&\\
\left(\mathsf{A}_1^3\right)^{(2)}&X_{\alpha_2}+X_{\alpha_3}+X_{\alpha_5}&64&19&\text{nc}&\\
\mathsf{A}_1^4&X_{\alpha_2}+X_{\alpha_3}+X_{\alpha_5}+X_{\alpha_7}&70&28&\text{nc}&\\
\hline\hline
\multicolumn6{c}{\text{Lie algebra of type }\mathsf{E}_8}\\
\hline
\mathsf{A}_1^2&X_{\alpha_1}+X_{\alpha_2}&92&14&\text{symp}&\\
\mathsf{A}_1^3&X_{\alpha_2}+X_{\alpha_3}+X_{\alpha_5}&112&31&\text{nc}&\\
\mathsf{A}_1^4&X_{\alpha_2}+X_{\alpha_3}+X_{\alpha_5}+X_{\alpha_7}&128&44&\text{nc}&\\
\hline\hline
\multicolumn6{c}{\text{Lie algebra of type }\mathsf{F}_4}\\
\hline
\widetilde{\mathsf{A}}_1&X_{\alpha_4}&22&7&\text{symp}&\\
\mathsf{A}_1\widetilde{\mathsf{A}}_1&X_{\alpha_1}+X_{\alpha_4}&28&10&\text{nc}&\\
\hline\hline
\multicolumn6{c}{\text{Lie algebra of type }\mathsf{G}_2}\\
\hline
{\widetilde{\mathsf{A}}}_1&X_{\alpha_1}&8&5&\text{nc}\\
\end{array}
\]
 
\medskip\noindent
By way of example, consider $x = X_{\alpha_4}$, a nilpotent element of type $\widetilde{\mathsf{A}}_1$ in $L$ of type $\mathsf{F}_4$.
The image  $\ad_x^2L$ is the $7$-dimensional subspace spanned by  the set $Y$ of all $X_\gamma$ for $\gamma$ running over \[\begin{array}{rcl}  &&
-\alpha_1-2\alpha_2-2\alpha_3 ,
-\alpha_1-\alpha_2-2\alpha_3 ,
-\alpha_2-2\alpha_3 ,
 \alpha_4  ,\\
&& \alpha_2+2\alpha_3 +2\alpha_4  ,
   \alpha_1 + \alpha_2+2\alpha_3 +2\alpha_4 ,
       \alpha_1 + 2\alpha_2+2\alpha_3 +2\alpha_4
     \end{array} \]
Computation shows that $[Y,Y] = 0$ and
 $ [Y,[Y,L]] =\ad_x^2L$. This implies that  $\ad_x^2L$ is an inner ideal. Since $x\in \ad_x^2L$, it coincides with $I_x$. All but its fourth element, which is $x$, are long root elements.
The elements  $X_{-\alpha_1-2\alpha_2-2\alpha_3}+ X_{\alpha_2+2\alpha_3 +2\alpha_4 }\pm X_{\alpha_4}$ are extremal and lie in $Y$. Their difference is $\pm2X_{\alpha_4}$. This establishes that the inner ideal $I_x$  is spanned by extremal elements. In fact, it is the span of the shadow on $\PExtr$ of an object of $\mathsf{B}(L)$ of type $4$.

For the remainder of the proof, we assume $\charac(k) = 3$. If the Levi subalgebra in which $x$ is distinguished has a type with a component distinct from $\mathsf{A}_1$ and  $\mathsf{A}_2$, then the ad-nilpotency index is greater than $3$, so the types of classes for $x$ that need to be considered are  $\mathsf{A}_2^r \mathsf{A}_1^s$  (up to tildes) with $r\ge0$ and $s\ge0$. Those with $r=0$ have been treated above. 
The table below lists those with $r\ge1$ with those  that have already been ruled out for a subalgebra appearing earlier in the table already omitted. 
The data in the table pertains to characteristic $3$, so the dimensions may differ from those for other characteristics. 
For instance,  in the last two cases of the Lie algebra of type $\mathsf{E}_6$, the dimension of $\ad_xL$  for  $p\ne2,3$ are $48$ and $54$, respectively, in accordance with the sizes of the centralizer groups listed in \cite[Table 22.2.3]{UniNil}, whereas they are $47$ and $51$, respectively, for $p=3$. Similarly, the dimension of $\ad_x^2L$ for the single case in the Lie algebra of type $\mathsf{E}_8$ is $164$ for  $p=3$; for $p>3$, it is $168$.

In all cases, the subspace $\ad_x^2L$ turns out to be noncommutative, so $I_x = L$. Since this noncommutativity persists for elements of Lie algebras embedded in larger Lie algebras, the representatives of elements occurring in smaller Lie algebras of the list are omitted. 
For instance, for the Lie algebra of type $\mathsf{E}_7$, there are $6$ cases, five of which occur in type $\mathsf{E}_6$, so we only need consider the type $\mathsf{A}_2\mathsf{A}_1^3$.
In characteristic $3$, the Lie algebra of type $\mathsf{G}_2$ is not simple (the quotient of type $\mathsf{A}_2$ has been dealt with in Proposition \ref{lmIOrd3EltSpByE}), so it does not appear in the list.
\[
\begin{array}{lcrrr}
\multicolumn5{c}{\text{characteristic }3}\\
\hline\hline
 \text{type}&\text{representative }x&\dim(\ad_xL)&\dim(\ad_x^2L)&\\
\hline\hline
\multicolumn5{c}{\text{Lie algebra of type }\mathsf{E}_6}\\
\hline
\mathsf{A}_2&X_{\alpha_1}+X_{\alpha_3}&42&21&\\
\mathsf{A}_2\mathsf{A}_1&X_{\alpha_1}+X_{\alpha_2}+X_{\alpha_3}&46&22&\\
\mathsf{A}_2\mathsf{A}_1^2&X_{\alpha_1}+X_{\alpha_2}+X_{\alpha_3}+X_{\alpha_5}&50&23&\\
\mathsf{A}_2^2&X_{\alpha_1}+X_{\alpha_2}+X_{\alpha_5}+X_{\alpha_6}&47&23&\\
\mathsf{A}_2^2\mathsf{A}_1&X_{\alpha_1}+X_{\alpha_2}+X_{\alpha_3}+X_{\alpha_5}+X_{\alpha_6}&51&24&\\
\hline\hline
\multicolumn5{c}{\text{Lie algebra of type }\mathsf{E}_7}\\
\hline
\mathsf{A}_2\mathsf{A}_1^3&X_{\alpha_1}+X_{\alpha_2}+X_{\alpha_3}+X_{\alpha_5}+X_{\alpha_7}&84&42&\\
\hline\hline
\multicolumn5{c}{\text{Lie algebra of type }\mathsf{E}_8}\\
\hline
\mathsf{A}_2^2\mathsf{A}_1^2&X_{\alpha_1}+X_{\alpha_2}+X_{\alpha_3}+X_{\alpha_5}+X_{\alpha_7}+X_{\alpha_8}&164&80&\\
\hline\hline
\multicolumn5{c}{\text{Lie algebra of type }\mathsf{F}_4}\\
\hline
\mathsf{A}_2&X_{\alpha_1}+X_{\alpha_2}&30&22&\\
\widetilde{\mathsf{A}}_2&X_{\alpha_3}+X_{\alpha_4}&30&22&\\
\mathsf{A}_2\widetilde{\mathsf{A}}_1&X_{\alpha_1}+X_{\alpha_2}+X_{\alpha_4}&34&19&\\
\widetilde{\mathsf{A}}_2{\mathsf{A}}_1&X_{\alpha_1}+X_{\alpha_3}+X_{\alpha_4}&36&23&\\
\end{array}
\]
The conclusion is that no other ad-nilpotent elements $x$ exists in $L$ for which $I_x$ is a proper inner ideal than those found for all characteristics distinct from $2$. This ends the proof of Proposition \ref{PropOrd3Exceptional}.
\end{proof}

The proof that inner ideals are spanned by extremal elements is complete. It enables us to study inner ideals by means of the geometry on pure extremal points. For the recognition of the geometries we will use two results.
The first is an updated version of a result in \cite{CohCoop}. Recall that $x^\perp$, for $x$ a point of a point-line space, denotes the set of all points collinear with $x$.
\begin{theorem} [Theorem 15.4.5 \cite{Shult}] \label{CCShult1545}Let $T$ be a strong parapolar space all of whose symplecta are polar spaces of the same rank $r\ge 3$ and all of whose singular subspaces have finite rank. Suppose, moreover, that, for every symplecton $S$ and every point $x$ of $T$ outside $S$, the singular subspace $x^\perp\cap S$ has singular rank distinct from $r-2$. Then we have one of the following six possibilities.
\begin{itemize}
\item $T$ is a polar space of rank $r$.
\item $r=3$ and $T$ is the Grassmannian $\mathsf{A}_{h,t}(\mathbb{D})$ of $(t+1)$-dimensional subspaces of a vector space $\mathbb{D}^{h+1}$, where $h$ and $t$ are natural numbers with $1\lt t\le \frac{h}{2}$ and $\mathbb{D}$ is a division ring. 
\item $r=3$ and $T$ is the quotient $\mathsf{A}_{2t-1,t}(\mathbb{D})/\langle\sigma\rangle$  of  $\mathsf{A}_{2t-1,t}(\mathbb{D})$ by a polarity $\sigma$ of $\mathbb{D}^{2t}$ of Witt index at most $t-5$, where $t$ is at least $5$.
\item $r=4$ and $T$  is a homomorphic image of the half-spin geometry $\mathsf{D}_{n,n}(\mathbb{F})$ by a polarity, where $n$ is a natural number with $n\ge5$ and $\mathbb{F}$ is a field. 
This homomorphism is an isomorphism if $n\le9$.
\item $r=5$ and $T$  is isomorphic to $\mathsf{E}_{6,1}(\mathbb{F})$, where $\mathbb{F}$ is a field.
\item  $r=6$ and $T$  is isomorphic to $\mathsf{E}_{7,7}(\mathbb{F})$, where $\mathbb{F}$ is a field.
\end{itemize}
\end{theorem}

The statement of the theorem in \cite{Shult}  includes the hypothesis that $T$ be locally connected, but here we do not mention it since it is implied by the requirement that $T$ be a strong parapolar space. It will be used for Lie algebras of type $\mathsf{E}_n$ $(n=6,7,8)$.

The other theorem we will use is for type  $\mathsf{F}_4$. It is also rephrased in current terminology.
\begin{theorem}[Theorem 2.3 \cite{Cohen8Meta}] \label{thMeta}
Let $P$ be a parapolar space of symplectic rank $3$ satisfying the following two properties.
 \begin{enumerate}[(i)]
\item The set $x^\perp\cap S$ is never a point if $S$ is a symplecton of $P$ and $x$ is a point a outside $S$.
\item There are no minimal $5$-circuits in the following sense: if $x_1,\ldots,x_5$ are distinct points of $P$ such that $x_i$ is collinear to $x_{i+1}$ and not to $x_{i+2}$ for each $i$ (indices taken modulo $5$), then there is an index $j\in\{1,\ldots,5\}$ such that $\{x_j, x_{j+2},x_{j+3}\}$ is contained in a symplecton. 
\end{enumerate}
Then $P$ is either a polar space of rank $3$ or a root shadow space of type $\mathsf{F}_{4,1}$.
\end{theorem}

We are now ready for the second step of the proof of Theorem \ref{th-InnerIsShadow}.

\begin{proof}
Let $I$ be a nontrivial proper inner ideal of $L$ and write $\PExtr_I =\{k\cdot x\mid x\in  E\cap I\}$ for the set of extremal points of $\PExtr$ contained in $I$.
By Propositions \ref{lmIOrd3EltSpByE} and \ref{PropOrd3Exceptional}, $I$ is spanned by $\PExtr_I $.

First suppose that $\FExtr=\emptyset$. If $\SExtr=\emptyset$, then there are no commuting pairs of distinct extremal points and so $I$ must be spanned by a single extremal point, which is the shadow of a flag of $\B(L)$, so the theorem obviously holds. If $\SExtr\ne\emptyset$, then, since $k$ is separably closed, Theorem \ref{RSSofL} gives that $\left(\PExtr,\SExtr\right)$ is a non-degenerate polar space and the Lie algebra has type $\mathsf{C}_n$. All shadows of the building on extremal points are singular subspaces of $\left(\PExtr,\SExtr\right)$. On the other hand, since strongly commuting pairs do not occur in $I$, the subspace $\PExtr_I$ of $(\PExtr,\SExtr)$  is singular, and so, by \cite[Proposition 11.5.14]{buek1}, it is the shadow of an object of $\B(L)$. Thus, we have dealt with the case where $\FExtr=\emptyset$.

For the remainder of the proof we assume $\FExtr\ne\emptyset$.  By Theorem \ref{RSSofL}, $\left(\PExtr,\FExtr\right)$ is the root shadow space of the spherical building $\B(L)$. If each pair of mutually distinct points of $\PExtr_I$ is strongly commuting, then  $\PExtr_I$ coincides with $\mathbb{P}(I)$ and is a projective subspace of  $(\PExtr,\FExtr)$,  so, again by \cite[Proposition 11.5.14]{buek1}, it is the shadow of a flag of $\B(L)$ (the proposition concerns only shadow spaces on a single node, but in the case of $\mathsf{A}_{n,\{1,n\}}$ the proof is also valid after a slight adaption).
Therefore, since $I$ is commutative (cf.~Theorem \ref{InnerIdealComm}), we may assume that there are extremal elements $x$, $y$ in $I$ such that $(k\cdot x,k\cdot y)$  is a polar pair. By Lemma \ref{InnerxyPolar}, the inner ideal  $I_{\{x,y\}} =\langle S(x,y)\rangle$ is contained in $I$. If $I_{\{x,y\}}$ and $I$ coincide, then $\PExtr_I=S(x,y)$ is a symplecton, which is known to be the shadow of a flag of $\B(L)$.
Therefore, we can restrict our attention to the case where, for each polar pair $(x,y)$ of extremal elements, $\PExtr_I$ strictly contains the symplecton $S( x,  y)$. This implies that $\PExtr_I$ is a subspace of $(\PExtr,\FExtr)$.
It is non-degenerate, since, if $z$ is a point of $\PExtr_I$ then $S(x,y)$ either contains $z$, in which case it contains a point at distance $2$ to $z$, or $x^\perp\cap S(x,y)$ is a singular subspace of $S(x,y)$ and so there must be points of $S(x,y)$ not collinear with $z$. This readily implies that $\PExtr_I$ is a convex closed subspace of $(\PExtr,\FExtr)$ that is strongly parapolar of diameter $2$.

Moreover, if $x$ is point of $\PExtr_I$ and $S$ a symplecton contained in $\PExtr_I$ but not containing $x$, then the set $x^\perp\cap S$ of points of $S$ collinear with $x$ is a singular subspace of $\PExtr_I$ of rank distinct from $0$ and, if $r\ge 4$, also from $r-2$, where $r$ is the rank of the symplecton. This is a fact concerning the root shadow space $(\PExtr,\FExtr)$ and, because $\PExtr_I$ is convex closed, it must also hold for $\PExtr_I$. In the cases where $L$ is of exceptional type, the symplectic rank $r$ is constant and one of $3,4,5,6$. In these cases, the isomorphism type of $\PExtr_I$ is as described in the conclusion of one of these two theorems.

We will be dealing with the different types of Lie algebra individually.

\medskip\noindent\emph{Type $\mathsf{A}_n$ $(n\ge 1)$}. The case $n=1$ has been dealt with above, so we take $n\ge2$. We will use the same setting as before and view $L$ as the quotient by the center of the Lie algebra $\sl(V)$, where $V$ is a vector space of dimension $n+1$ over $k$.  For the time being, we will set aside the case $(n,p)=(2,3)$. Then 
each extremal element of $\sl(V)$ is a nilpotent linear map $X : V\to V$ of index $2$ and rank $1$. Thus, its image and kernel are an incident pair of a point $p_X$ and a hyperplane $h_X$ of $\ProjSp(V)$. The extremal point $k\cdot X$ of $\PExtr(\sl(V))$, as well as its image in $\PExtr(L)$, corresponds bijectively to the incident pair $(p_X,h_X)$. In these terms, collinearity between two extremal elements $X$ and $Y$ means that $p_X=p_Y$ or $h_X= h_Y$, and commuting means that $p_X,p_Y\in h_X\cap h_Y$.
Thus, if  $X$ and $Y$ are non-collinear members of $\PExtr_I$, then $p_X\ne p_Y$ and $h_X\ne h_Y$, and the symplecton
 generated by them consists of all $Z$ with $p_Z\in \langle p_X,p_Y\rangle$ and $h_Z\supseteq h_X\cap h_Y$. Continuing with more members of $\PExtr_I$, we find that there are linear subspaces $U$ and $W$ of $V$ with $U\subseteq W$ such that $\PExtr_I$ consists of all pairs $(p,H)$ with $p\in U$ and $H\supseteq W$. This means that, with $s$ and $t$  the dimension of $U$ and $W$, respectively, the  subspace $\PExtr_I$ is the shadow of a flag of type $\{s,t\}$  of the building $\B(L)$.  
 
To finish the discussion of type $\mathsf{A}_n$, we discuss the expectional case where $(n,p) = (2,3)$. For fields (not necessarily $p=3$), the extremal points coming from nilpotent maps of index $2$ and rank $1$ correspond to the chambers of the building of type $\mathsf{A}_2$. If $p=3$, another class of extremal elements emerges. It consists of the nilpotent maps  $V\to V$ of index $3$ and rank $2$. Since all nilpotent elements of $\sl(V)$ of ad-nilpotency index $3$ are extremal, the projective points of each proper inner ideal are extremal, and so $\PExtr_I$ must be singular. In other words, no symplecta occur, that is $\SExtr=\emptyset$. This implies that $(\PExtr,\FExtr)$ is a generalized hexagon, so $\B(L)$ is of type $\mathsf{G}_2$ and the theorem follows from the above observations. This is in accordance with the fact that the automorphism group of $L$ is a split algebraic group of type $\mathsf{G}_2$.

\medskip\noindent\emph{Types $\mathsf{B}_n$ $(n\ge 2)$ and $\mathsf{D}_n$ $(n\ge 4)$}. Here, we use $V$ and $B$ as in Lemma \ref{lmOrth}, in order to represent $L$ as the quotient by the center of $\sl(V)_B$. The set $\PExtr(L)$ corresponds to the set of lines of the polar space $P$ of isotropic points of $\ProjSp(V)$ (that is, those $k\cdot v$  for $v\in V$ with $B(v,v) = 0$). Explicitly, if $X\in \Extr(\sl(V)_B)$, then there are isotropic vectors $a$, $b$ in $V$ with $B(a,b) = 0$  such that $X= D_{a,b}$ as described in Lemma \ref{lmOrth}, and the extremal point $k\cdot X$ corresponds to the isotropic line $\ell_X := \langle a , b\rangle $ of $P$.  The polar space $P$ is non-degenerate and has rank $n$. We recall that two extremal points $k\cdot X$ and $k\cdot Y$ in $\Extr$ are
\begin{itemize}
\item collinear if and only if $\ell_X$ and $\ell_Y$ span a singular plane of $P$ and have a point in common,
\item polar if and only if they are not collinear and either meet, in which case $S( X,Y) $ is the set of lines on a point of $P$ (a symplecton of rank $n-1$) or span a singular subspace of $P$ of rank $3$, in which case $S( X, Y) $ corresponds to the set of all lines in the singular subspace (a symplecton of rank $3$),
\item special if and only if the lines do not meet, are not contained in a singular subspace of $P$, and there is a unique line of $P$ in $\ell_X^\perp\cap\ell_Y^\perp$.
\end{itemize}
Inspection of the Dynkin diagrams of type $\mathsf{B}_n$  and $\mathsf{D}_n$ shows that the minimal flags of the building $\text{B}(L)$ whose shadows on $\PExtr$ are not singular subspaces have the following types.
\begin{itemize}
\item $\{1\}$. In terms of $P$, the shadows are the symplecta of lines of $P$ containing a point of $P$.
\item $\{j\}$ for $j\in\{3,4,\ldots,n\}$. In terms of $P$, the shadows are the Grassmannians of all lines in a singular subspace of $P$.
\end{itemize}

We continue with the assumption that $\PExtr_I$ strictly contains a symplecton $S( X,  Y)$, where $(X,Y)$ is a polar pair in $\Extr(\sl(V))$.
Assume $k\cdot Z\in\PExtr_I \setminus S( X,  Y)$. If $\ell_X$ and $\ell_Y$ meet, then $S(X, Y)$ corresponds to the set of all lines on a point $y$ of $P$, and $\ell_Z$ is a line not contained in $y^\perp$. 
Take a point $q$ on $\ell_Z\setminus y^\perp$. The subspace $q^\perp\cap y^\perp$ is a non-degenerate polar space, so there is a point $r\in q^\perp\cap y^\perp\setminus \ell_z^\perp$. This point is collinear with $y$ and distinct from it, so $yr $ is a line of $P$. Let $U\in\Extr(\sl(V))$ be such that $\ell_U=yr$. This line does not lie in a singular subspace with $\ell_Z$ and does not meet  $\ell_Z$, so  $(U,Z)$ is not commutative, a contradiction. 

We conclude that the span of  $\ell_X$ and $\ell_Y$ in $P$ must be a singular subspace for every polar pair $(k\cdot X,k\cdot Y)$ in $\PExtr_I$. In particular, there is a singular subspace $R$ of $P$ such that $k\cdot W\in \PExtr_I$ whenever $\ell_W\in R$. Choose $R$ maximal with respect to this property. The assumption on $Z$ implies that $\ell_Z$ does not lie in $R$. Since $\ell_Z\subset R^\perp$  contradicts the maximality of $R$, the line $\ell_Z$ contains a point $q$ that is not collinear with all of $R$. Take $r\in R\setminus q^\perp$, so $r\not\in \ell_Z$. Since the singular rank of $R$ is at least $2$ (due to the presence of a symplecton), there is a line $m$ in $R$ on $r$ with $\ell_Z\cap R$ (which is empty or a singleton) not contained in $m$. Now $m=\ell_U$ for some $U\in \Extr(\sl(V))$ and, since $\ell_U$ and $\ell_Z$ do not meet and do not span a singular subspace of $P$,  the pair $(U,Z)$ of elements from $I$ does not commute. This contradiction with $I$ being commutative shows that  $ \PExtr_I$ cannot be strictly larger than a symplecton. This ends the proof of the theorem for the orthogonal cases.

Since type $\mathsf{C}_n$ has already been dealt with, we continue with the exceptional types.

\medskip\noindent\emph{Type $\mathsf{E}_6$}. The subspace $\PExtr_I$ satisfies the conditions of Theorem \ref{CCShult1545} with $r=4$. This gives
that is  $\PExtr_I$  is a homomorphic image of a half-spin geometry of type $\mathsf{D}_{t,t}(\mathbb{F})$ by a polarity, where $t$ is a natural number with $t\ge5$ and $\mathbb{F}$ is a field. Maximal singular subspaces of $(\PExtr,\FExtr)$ have rank $4$ and those of the homomorphic images mentioned in the theorem have rank $t-1$, so we must have $t=5$. Since the singular subspaces are projective spaces defined over $k$, we have $\mathbb{F} = k$. Therefore, $\PExtr_I\isom \mathsf{D}_{5,5}(k)$. Since $S(x,y)$ is properly contained in the convex closed space $\PExtr_I$ of diameter $2$, there is an extremal point $x\in\PExtr_I\setminus S$ such $x^\perp\cap S$ is a maximal singular subspace of $S$. The convex closure of $S\cup\{x\}$ is contained in the shadow $Q$ of an object of type $1$ or $6$, which must be contained in $\PExtr_I$.  By the same theorem as for $\PExtr_I$, the space $Q$ is also isomorphic to $\mathsf{D}_{5,5}(k)$. The maximal singular subspaces of $Q$ and $\PExtr$ must coincide, and so $Q=\PExtr_I$. This establishes that $\PExtr_I$ coincides with the shadow of a flag of type  $\{1\}$ or $\{6\}$ of $\B(L)$, and so $I$ is the span of this shadow. Thus, the theorem holds for $L$ of  type $\mathsf{E}_6$.

\medskip\noindent\emph{Type $\mathsf{E}_7$}. The subspace $\PExtr_I$ satisfies the conditions of Theorem \ref{CCShult1545} with $r=5$, so $\PExtr_I$  is isomorphic to $\mathsf{E}_{6,1}(\mathbb{F})$, where $\mathbb{F}$ is a field. Since the lines of $\PExtr_I$ are lines of $\ProjSp(L)$, we must have $\mathbb{F} = k$.  On the other hand, as for the previous case, the fact that $\PExtr_I$ is convex closed and strictly contains $S$ implies the existence of an element $x\in \PExtr_I\setminus S$ with $x^\perp\cap S$ a subspace of maximal singular rank in $S$. Again the convex closure of $S\cup\{x\}$ is a shadow space $Q$ isomorphic to  $\mathsf{E}_{6,1}(k)$ contained in $\PExtr_I$, and we conclude again that $\PExtr=Q$, the shadow of a flag of type $\{7\}$. This settles the theorem for type $\mathsf{E}_7$.

\medskip\noindent\emph{Type $\mathsf{E}_8$}. Application of Theorem \ref{CCShult1545}  with $r=7$ gives that  $\PExtr_I$ must be a polar space, a contradiction. So there are no proper inner ideals other than those spanned by shadows of flags of the building.

\medskip\noindent\emph{Type $\mathsf{F}_4$}. Application of Theorem \ref{thMeta}  with $r=3$ enables us to conclude that $\PExtr_I$ must coincide with $\PExtr$.

\medskip\noindent\emph{Type $\mathsf{G}_2$}.  Since $p=3$ has already been dealt with under $\mathsf{A}_2$, we assume $p>3$. As a consequence of the study of elements occurring in $I$ in Proposition \ref{lmIOrd3EltSpByE}, the only proper inner ideal $I_x$ for an ad-nilpotent element $x$ occurs for $x$ a long root element of $L$. In particular, each inner ideal consists of $0$ and extremal elements only, and so corresponds to a singular subspace, of dimension 1 or 2, of $L$. These are the spans of the shadows of the objects of type $2$ and $1$, respectively, of $\B(L)$.
  
This ends the proof of Theorem \ref{th-InnerIsShadow}.
\end{proof}

\section{Conclusion}\label{S-conclusion}

We finish the proof of the main theorem  by combining Proposition \ref{ShadowsInnerProposition} and Theorem \ref{th-InnerIsShadow}, 

\begin{proof}
Let $\Lie(G)$ be the Lie algebra of a simple algebraic group $G$ defined over a field $k$ of characteristic distinct from $2$. Let $L_G$ be the unique nontrivial simple quotient of $[\Lie(G),\Lie(G)]$. If $\charac(k)=3$, assume that $G$ is not of absolute type $\mathsf{A}_2$.   Let $k_s$ be the separable closure of $k$ with Galois group $\Gamma$.

Now suppose that $I$ is a proper nontrivial inner ideal of $L_G$ and write $L = L_G$. Then $I\otimes k_s$ is an inner ideal of $L\otimes k_s$. Since $k_s$ is a splitting field for $G$ over $k$, the Lie algebra $L\otimes k_s$ is spanned by its pure extremal elements. By Theorem \ref{th-InnerIsShadow}, the set $T=\PExtr(L\otimes k_s)_{I\otimes k_s}$ is the shadow of a flag of the building $\B(L\otimes k_s)$ such that $I\otimes k_s =\langle T\rangle$. Now $\Gamma$ leaves $I\otimes k_s$ and $T$ invariant, and $I = (I\otimes k_s)^\Gamma = \langle
T\rangle^\Gamma$ is the set of $\Gamma$-fixed points of the span of the shadow $T$ of a flag of $\B(L\otimes k_s)$ fixed by $\Gamma$. 

Conversely, let $R$ be a $\Gamma$-invariant shadow  on $\PExtr(L\otimes k_s)$ of a flag of $\B(L\otimes k_s)$. The subspace $\langle R\rangle$ of $L\otimes k_s$ is $\Gamma$-invariant, and so $I = \langle R\rangle \cap L$ is a subspace of $L$ of the same dimension. By Proposition \ref{ShadowsInnerProposition}, the subspace $\langle R\rangle$ is a proper inner ideal, so $[I,[I,L]]\subseteq L\cap( I\otimes k_s) = I$, which shows that $I$ is a proper inner ideal of $L$. The map sending $I$ to $T =\PExtr(L\otimes k_s)_{I\otimes k_s}$ and the map sending the $\Gamma$-invariant shadow $R$ of a $\Gamma$-fixed flag of $\B(L\otimes k_s)$ on $\PExtr(L\otimes k_s)$ to $ \langle R\rangle \cap L$ are each other's inverses, so these establish the required bijective correspondence.
\end{proof}

\bibliographystyle{plain}
\def\cprime{$'$} \def\Dbar{\leavevmode\lower.6gex\hbox to 0pt{\hskip-.23ex
  \accent16\hss}D}

\end{document}